\documentclass[reqno]{amsart}

\usepackage{amsmath}
\usepackage{amsfonts}
\usepackage{amsthm}
\usepackage{amssymb}
\usepackage{graphicx}
\usepackage{graphics}
\usepackage{bm}
\usepackage{dsfont}
\usepackage{color}
\usepackage[font=footnotesize]{caption}
\usepackage[all]{xy}

\usepackage{url}
\usepackage{enumitem, hyperref}\hypersetup{colorlinks}

\definecolor{darkred}{rgb}{1,0,0} 
\definecolor{darkgreen}{rgb}{0,0.8,0}
\definecolor{darkblue}{rgb}{0,0,1}

\hypersetup{colorlinks,
linkcolor=darkblue,
filecolor=darkgreen,
urlcolor=darkred,
citecolor=darkgreen}

\numberwithin{equation}{section}

\newtheoremstyle{personal}%
{12pt}
{12pt}
{\slshape}
{}
{\bfseries}
{.}
{.5em}
{}
\theoremstyle{personal}%
\newtheorem{thm}{Theorem}[section]
\newtheorem{cor}[thm]{Corollary}
\newtheorem{lem}[thm]{Lemma}
\newtheorem{prop}[thm]{Proposition}
\newtheorem{quest}{Question}[section]
\theoremstyle{definition}

\newtheorem{exm}[thm]{Example}

\DeclareRobustCommand{\llongtwoheadleftarrow}{\twoheadleftarrow\joinrel\relbar\joinrel\relbar}
\DeclareMathOperator*{\bssurjup}{\llongtwoheadleftarrow} 
\newcommand{\XX}{\mathcal{X}}
\newcommand{\Exp}{\mathrm{Exp}}
\newcommand{\dist}{\mathrm{dist}}
\newcommand{\CaP}{\mathrm{Ca}\PP}
\newcommand{\CP}{\C\PP}
\newcommand{\HP}{\mathds{H}\PP}
\newcommand{\Lie}{\mathcal{L}}
\newcommand{\HH}{\mathcal{H}}
\newcommand{\eh}{\ind_{\mathrm{EH}}}
\newcommand{\parity}{\mathrm{par}}
\newcommand{\Sp}{\mathrm{Sp}}
\newcommand{\fix}{\mathrm{fix}}
\newcommand{\cz}{\ind_{CZ}}
\newcommand{\pr}{\mathrm{pr}}
\newcommand{\A}{\mathcal{A}}

\newcommand{\fr}{\mathrm{ind_{FR}}}

\newcommand{\N}{\mathds{N}}

\newcommand{\Z}{\mathds{Z}}
\newcommand{\R}{\mathds{R}}
\newcommand{\Q}{\mathds{Q}}

\newcommand{\RP}{\mathds{R}\PP}

\newcommand{\PP}{\mathds{P}}
\newcommand{\C}{\mathds{C}}

\newcommand{\NN}{\mathcal{N}}
\newcommand{\WW}{\mathcal{W}}
\newcommand{\UU}{\mathcal{U}}
\newcommand{\KK}{\mathcal{K}}

\newcommand{\crit}{\mathrm{crit}}
\newcommand{\diff}{\mathrm{d}}

\newcommand{\ind}{\mathrm{ind}}
\newcommand{\id}{\mathrm{id}}
\newcommand{\nul}{\mathrm{nul}}
\newcommand{\sys}{\mathrm{sys}}

\newcommand{\injrad}{\mathrm{injrad}}
\newcommand{\Tan}{\mathrm{T}}

\DeclareRobustCommand{\llonghookrightarrow}{\lhook\joinrel\relbar\joinrel\relbar\joinrel\rightarrow}
\DeclareRobustCommand{\llongrightarrow}{\relbar\joinrel\relbar\joinrel\rightarrow}
 
\DeclareMathOperator{\supp}{\mathrm{supp}} 
\DeclareMathOperator*{\toup}{\longrightarrow} 
\DeclareMathOperator*{\ttoup}{\llongrightarrow} 
 
\DeclareMathOperator*{\eembup}{\llonghookrightarrow}

\begin{document}

\title[On the spectral characterization of Besse and Zoll Reeb flows]{On the spectral characterization\\ of Besse and Zoll Reeb flows}

\author{Viktor L.\ Ginzburg}
\address{Viktor Ginzburg\newline\indent Department of Mathematics, UC Santa Cruz, Santa Cruz, CA 95064, USA}
\email{ginzburg@ucsc.edu}

\author{Ba\c{s}ak Z.\ G\"urel}
\address{Ba\c{s}ak G\"urel\newline\indent Department of Mathematics, UCF, Orlando, FL 32816, USA}
\email{basak.gurel@ucf.edu}

\author{Marco Mazzucchelli}
\address{Marco Mazzucchelli\newline\indent CNRS, UMPA, \'Ecole Normale Sup\'erieure de Lyon, 69364 Lyon, France}
\email{marco.mazzucchelli@ens-lyon.fr}

\date{September 7, 2019. \emph{Revised}: March 23, 2020.}



\thanks{This work was partially supported by the NSF Grant DMS-1440140 via MSRI (BG and MM), the NSF CAREER award DMS-1454342 (BG), and  by Simons Foundation Collaboration Grant 581382 (VG)}

\begin{abstract}
A closed contact manifold is called Besse when all its Reeb orbits are closed, and Zoll when they have the same minimal period. In this paper, we provide a characterization of Besse contact forms  for convex contact spheres and Riemannian unit tangent bundles in terms of $S^1$-equivariant spectral invariants. Furthermore, for restricted contact type hypersurfaces of symplectic vector spaces, we give a sufficient condition for the Besse property via the Ekeland--Hofer capacities.

\tableofcontents
\end{abstract}

\maketitle

\vspace{-40pt}

\section{Introduction}
\label{s:introduction}

This work gives spectral characterizations of several classes of higher dimensional Besse and Zoll contact manifolds. Recall that a contact manifold is a pair $(\Sigma,\lambda)$, where $\Sigma$ is a manifold of dimension $2n-1$, and $\lambda$ is a 1-form on $\Sigma$ such that $\lambda\wedge(\diff\lambda)^{n-1}$ is nowhere vanishing. Throughout the paper, all contact manifolds are assumed to be closed and connected. The spectrum mentioned above is the one related to the Reeb dynamics. 
Namely, denote by $R$ the Reeb vector field on $(\Sigma,\lambda)$, defined by $\lambda(R)\equiv1$ and $\diff\lambda(R,\cdot)\equiv0$. 
The action spectrum $\sigma(\Sigma,\lambda)$ is the set of periods of the closed Reeb orbits, i.e., of the periodic orbits of the flow $\phi_R^t:\Sigma\to\Sigma$ of $R$. Here by period we mean any, not necessarily minimal, period of a Reeb orbit: a closed Reeb orbit with minimal period $\tau>0$ contributes to the spectrum all multiples of $\tau$. In all  settings that we consider, the action spectrum is known to be non-empty, whereas in general this is a conjecture due to Weinstein.

A closed connected contact manifold is called \textbf{Besse} when all its Reeb orbits are closed, and in such a case the Reeb orbits admit a common period by a theorem of Wadsley, \cite{Wadsley:1975sp}. When the common period of the Reeb orbits is also their minimal period, the contact manifold is called \textbf{Zoll}\footnote{In the literature, Besse and Zoll contact manifolds are sometimes called almost regular and regular contact manifolds respectively, see, e.g., \cite{Thomas:1976aa}.}. In other words, these are the prequantization bundles equipped with connection forms. To be more specific, consider a symplectic manifold $(B,\omega)$ such that $[\omega]$  is integral, i.e., $[\omega]$ lies in the image of  $H^2(B;\Z)$ in $H^2(B;\R)$. There exists an $S^1$-bundle $\pi:\Sigma\to B$ with first Chern class $-[\omega]$ or, more precisely, such that $-[\omega]$ is the image of its first Chern class. This circle bundle is not unique if $H^2(B;\Z)$ has torsion. Identifying the Lie algebra of $S^1$ with $\R$, pick a connection form $\lambda$ on this $S^1$-bundle. (We refer the reader to, e.g., \cite{Guillemin:2002aa} for a discussion of sign conventions used here.) The connection form $\lambda$ is an $S^1$-invariant form on $\Sigma$ such that $d\lambda=\pi^*\omega$. Note that such a form is not unique up to a gauge transformation if $H^1(B;\R)\neq 0$. It is easy to see that $\lambda$ is a contact form. The contact manifold $(\Sigma,\lambda)$ is often referred to as a prequantization or Boothby--Wang bundle \cite{Boothby:1958ss} and it is not hard to show that every Zoll contact manifold has this form.

In \cite{Mazzucchelli:2018aa}, the third author and Suhr proved that the unit tangent bundle of a Riemannian 2-sphere $(S^2,g)$, equipped with its standard contact form, is Zoll if and only if the simply closed geodesics of $(S^2,g)$ have the same length.
In the recent paper \cite{Cristofaro-Gardiner:2019aa}, Cristofaro-Gardiner and the third author proved that a closed connected contact 3-manifold $(\Sigma,\lambda)$ is Besse if and only if its action spectrum $\sigma(\Sigma,\lambda)$ has rank one, that is, $\sigma(\Sigma,\lambda)\subset\{m\tau\ |\ m\in\Z\}$ for some $\tau>0$. The proof of this result uses in an essential way the properties of the spectral invariants from embedded contact homology (ECH), \cite{Hutchings:2014qf, Cristofaro-Gardiner:2016rp}, a tool which is only available in dimension 3. 

Even though a characterization such as the one in \cite{Cristofaro-Gardiner:2019aa} seems currently out of reach for general high dimensional contact manifolds, in this paper we provide results in this direction for three classes of contact manifolds: convex hypersurfaces of symplectic vector spaces $(\R^{2n},\omega)$, more general restricted contact type hypersurfaces of $(\R^{2n},\omega)$, and Riemannian unit tangent bundles. With one notable exception (Theorem~\ref{t:Zoll_pinched}), we are not able to obtain characterizations of the Besse property from the mere knowledge of the action spectrum; our criteria require the knowledge of suitable $S^1$-equivariant spectral invariants, which are selectors of values from the action spectrum.

\subsection{Convex contact spheres}
\label{ss:convex}
Consider $\R^{2n}$ equipped with its standard symplectic form 
\begin{align*}
 \omega=\sum_{i=1}^n \diff x_i\wedge\diff y_i.
\end{align*}
Let $C\subset\R^{2n}$ be a compact ball containing the origin in its interior and having a smooth boundary $\Sigma=\partial C$ which is strongly convex, i.e., it is a level-set of a smooth function with positive definite Hessian at every point. The primitive 
\begin{align*}
\sum_{i=1}^n\tfrac{1}{2}\big(x_i\,\diff y_i-y_i\,\diff x_i\big)
\end{align*}
of  $\omega$ restricts to a contact form $\lambda$ on $\Sigma$. In what follows, we will refer to such a hypersurface $\Sigma$ as a convex contact sphere, implicitly equipped with its contact form $\lambda$, and denote its action spectrum by $\sigma(\Sigma)$.

In their seminal work \cite{Ekeland:1987aa}, Ekeland and Hofer constructed an increasing sequence
$$
c_0(\Sigma)\leq c_1(\Sigma)\leq c_2(\Sigma)\leq \ldots
$$
 of spectral invariants $c_i(\Sigma)\in\sigma(\Sigma)$ for integers $i\geq0$ by means of a suitable min-max procedure with $S^1$-equivariant cohomology classes. In this paper, we recall an alternative definition of these spectral invariants relying on Clarke's dual action functional, as in Ekeland's monograph \cite[Chapter~V]{Ekeland:1990lc};  see Section~\ref{ss:convex_spectral_invariants}. 
 
Every $\tau$-periodic Reeb orbit of $\Sigma$ has a well defined Conley-Zehnder index (see Section~\ref{ss:Morse_indices_Clarke}). In a more general symplectic setting, the Conley-Zehnder indices can have arbitrary sign, but on a convex contact sphere they are always larger than or equal to $n$.
If $\Sigma$ is Besse and $\tau$ is a common period for its Reeb orbits, then all $\tau$-periodic Reeb orbits have the same Conley-Zehnder index $\mu\in\N$, and we say that $\Sigma$ is $\bm{(\tau,\mu)}$\textbf{-Besse}.

\begin{exm}
\label{ex:ellipsoids}
Consider the ellipsoid
\begin{align*}
 E(a)=\left\{ z=(z_1,...,z_n)\in\R^{2n}\ \left|\ \sum_{h=1}^n \frac{|z_h|^2}{a_h}=\frac1\pi \right.\right\},
\end{align*}
where $a=(a_1,...,a_n)$ and $0<a_1\leq...\leq a_n<\infty$. The associated Reeb flow is the linear one
\begin{align*}
\phi_R^t:E(a)\to E(a),
\qquad
\phi_R^t(z)=(e^{J2\pi t/a_1}z_1,...,e^{J2\pi t/a_n}z_n ),
\end{align*}
where $J$ is the standard complex structure of the symplectic vector space $(\R^{2n},\omega_0)$. The ellipsoid $E(a)$ is Besse if and only if the ratios $a_h/a_k$ are rational for all $h,k\in\{1,...,n\}$. In this case, the minimal period $\tau_0>0$ of the Reeb flow is precisely the least common multiple of $a_1,...,a_n$. For any multiple $\tau>0$ of $\tau_0$, the $\tau$-periodic Reeb orbits have Conley-Zehnder index $\mu = 2\big(\tfrac{\tau}{a_1}+...+\tfrac{\tau}{a_n}\big) - n$, and the spectral invariants satisfy 
\begin{align}
\label{e:equality_Besse_ellipsoids}
\tau=c_i(\Sigma)=c_{i+n-1}(\Sigma) 
\end{align}
for $i=(\mu-n)/2$, see Section~\ref{ss:ellipsoids}.
\hfill\qed
\end{exm}

Our first result shows that the equality~\eqref{e:equality_Besse_ellipsoids} is indeed a characterizing property of $(\tau,\mu)$-Besse convex contact spheres.

\begin{thm}\label{t:Besse_convex}
A convex contact sphere $\Sigma\subset\R^{2n}$ is Besse if and only if $c_i(\Sigma)=c_{i+n-1}(\Sigma)$ for some integer $i\geq0$. In this case, $\Sigma$ is $(\tau,\mu)$-Besse for $\tau=c_i(\Sigma)=c_{i+n-1}(\Sigma)$ and $\mu=2i+n$.
\end{thm}

An immediate consequence of Theorem~\ref{t:Besse_convex} is a characterization of  Zoll convex contact spheres:

\begin{cor}\label{c:Zoll_convex}
A convex contact sphere $\Sigma\subset\R^{2n}$ is Zoll if and only if $c_0(\Sigma)=c_{n-1}(\Sigma)$.
\end{cor}

A convex contact sphere $\Sigma=\partial C\subset\R^{2n}$ is said to be $\delta$-pinched when it is squeezed between two round balls $B^{2n}(r)\subseteq C\subseteq B^{2n}(R)$ whose radii have ratio $R/r<\delta$. The class of $\sqrt{2}$-pinched convex contact spheres $\Sigma\subset\R^{2n}$ is particularly significant in the study of periodic orbits: for instance, a theorem due to Ekeland and Lasry \cite{Ekeland:1980aa} asserts that they always have at least $n$ closed Reeb orbits. (The multiplicity problem is still open for general convex contact spheres of dimension at least $9$ without any additional non-degeneracy assumptions; see however \cite{Long:2002aa}, and also \cite{Ginzburg:2020aa} for more recent results and further references.) Building on \cite{Ekeland:1980aa}, we provide a characterization of Zoll $\sqrt{2}$-pinched convex contact spheres, which is solely based on the action spectrum. Recall that the systole $\sys(\Sigma)$ is the minimum of the spectrum $\sigma(\Sigma)$, that is, the minimum among all the periods of the closed Reeb orbits of $\Sigma$.

\begin{thm}\label{t:Zoll_pinched}
A convex contact $\delta$-pinched sphere $\Sigma\subset\R^{2n}$ with $\delta\in(1,\sqrt2]$ is Zoll if and only if its action spectrum satisfies 
$\sigma(\Sigma)\cap(\sys(\Sigma),\delta^2\sys(\Sigma))=\varnothing$. 
\end{thm}
In the context of geodesic flows (cf.~Subsection~\ref{ss:geodesic_flows}), a theorem in the same spirit was proved by Ballmann--Thorbergsson--Ziller \cite[Theorem~A]{Ballmann:1983fv}.

\subsection{Restricted contact type hypersurfaces of symplectic vector spaces}
We now consider a larger class of closed hypersurfaces 
$\Sigma\subset\R^{2n}$ formed by hypersurfaces of restricted contact type, i.e., such that there exists a primitive $\Lambda$ of the standard symplectic form $\omega$ which restricts to a contact form on $\Sigma$. For instance, the boundary of a star-shaped domain in $\R^{2n}$ has restricted contact type. In fact, every contact form on $S^{2n-1}$ supporting the standard contact structure can be obtained by a star-shaped embedding. The action spectrum $\sigma(\Sigma,\Lambda|_\Sigma)$ is independent of the choice of the primitive $\Lambda$, and therefore will be simply denoted by $\sigma(\Sigma)$. Analogously, if $(\Sigma,\Lambda|_\Sigma)$ is Besse, the same will be true if we replace $\Lambda$ by any other primitive of $\omega$ which restricts to a contact form on $\Sigma$, and therefore we will simply say that $\Sigma$ is Besse.
For this class of contact manifolds, the variational theory of closed Reeb orbits is more involved than in the convex case, and several approaches (predating Floer theory and symplectic homology) were developed by Viterbo \cite{Viterbo:1989aa}, by Hofer and Zehnder \cite{Hofer:1987aa}, and by Ekeland and Hofer \cite{Ekeland:1989aa, Ekeland:1990aa}. In this section, we adopt the latter approach, which involves the Ekeland--Hofer symplectic capacities. 

For each integer $i\geq0$ and for each bounded subset $B\subset\R^{2n}$, we denote by $c_i(B)$ the $i$-th Ekeland--Hofer capacity. This is again an increasing sequence:
\[
c_0(B)\leq c_1(B)\leq c_2(B)\leq \dotsc 
\]

Two  terminological remarks are due. First, note that there is a potential conflict of notation here, since $c_i(\cdot)$ was already used in Subsection~\ref{ss:convex} to denote the spectral invariants for the Clarke's dual action functional. Nevertheless, recent papers of Abbondandolo and Kang \cite{Abbondandolo:2019aa}, and Irie \cite{Irie:2019aa} seem to suggest that the $i$-th Ekeland--Hofer capacity of a convex contact sphere may indeed coincide with its $i$-th spectral invariant defined via the Clarke's dual action functional.  Secondly, here we think of $c_i$ as a function of $B$ and thus $c_i$ naturally extends to a function of compact subsets of $\R^{2n}$. In particular, one has $c_i(\Sigma)$ for $\Sigma=\partial B$ defined in this way. This notation could be somewhat inconsistent with the one from Section~\ref{ss:convex}, where the spectral invariants, a.k.a.\ action selectors, a.k.a.\ capacities, are associated with the domain bounded by $\Sigma$. However, one of the remarkable features of the Ekeland--Hofer capacities is that they are action selectors: if $B\subset\R^{2n}$ is a compact subset whose boundary $\Sigma=\partial B$ is smooth and of restricted contact type, then $c_i(B)=c_i(\Sigma)\in\sigma(\Sigma)$ for all $i\geq0$. 

Here we prove that the Ekeland--Hofer capacities provide sufficient conditions for a restricted contact type hypersurface $\Sigma$ to be Besse.

\begin{thm}
\label{t:EH_Besse}
Let $\Sigma\subset\R^{2n}$ be a closed hypersurface of restricted contact type with discrete action spectrum $\sigma(\Sigma)$. Assume that $c_{i}(\Sigma)=c_{i+n-1}(\Sigma)$ for some integer $i\geq0$. Then $\Sigma$ is Besse and $c_i(\Sigma)$ is a common period of its closed Reeb orbits.
\end{thm}

In this theorem, the assumption that the action spectrum be discrete does not seem to be essential, and is imposed only to avoid technical difficulties. Up to this assumption, Theorem~\ref{t:EH_Besse} is a generalization of one direction -- the ``if" assertion -- of Theorem~\ref{t:Besse_convex}. We do not know whether the Ekeland--Hofer capacities actually provide a characterization of restricted contact type hypersurfaces, although this is likely to be the case. 

\begin{quest}
Does the converse implication in Theorem~\ref{t:EH_Besse} hold? More precisely, if $\Sigma$ is a Besse, closed, restricted contact type hypersurface whose Reeb orbits have common period $\tau$, is it true that $c_{i}(\Sigma)=c_{i+n-1}(\Sigma)=\tau$, where $2i+n$ is the Conley--Zehnder index of any $\tau$-periodic Reeb orbit?
\end{quest}

\subsection{Geodesic flows}
\label{ss:geodesic_flows}

The last class of contact manifolds which we consider is that of Riemannian unit tangent bundles 
\[
\Sigma=SM=\big\{(q,v)\in\Tan M\ \big|\ \|v\|_g=1\big\},
\]
where $M$ is a closed manifold of dimension $n\geq2$, and $g$ is a Riemannian metric on $M$. The unit tangent bundle $SM$ is implicitly equipped with the contact form which is the restriction of the Liouville 1-form on $\Tan M$. The Reeb flow on $SM$ is the geodesic flow of $(M,g)$, and in particular the closed Reeb orbits are the lifts of closed geodesics of $(M,g)$.

A Riemannian metric $g$ on $M$ is called Besse (resp.\ Zoll) when its unit tangent bundle $SM$ is Besse (resp.\ Zoll). We will say that a closed manifold $M$ of dimension at least $2$ is Besse (resp.\ Zoll) when it admits a Besse (resp.\ Zoll) Riemannian metric. In the recent paper \cite{Mazzucchelli:2018pb}, the third author and Suhr provided a characterization of the Zoll property of a Riemannian metric in terms of suitably defined spectral invariants associated with a finite-dimensional reduction of the classical energy functional. Here, we provide an analogous characterization in terms of the usual $S^1$-equivariant spectral invariants of geodesic flows.

A celebrated theorem due to Bott and Samelson \cite{Bott:1954aa, Samelson:1963aa, Besse:1978pr} asserts that any Besse manifold has the same integral cohomology ring as one of the compact rank-one symmetric spaces
\begin{align*}
S^n,\ \RP^n,\ \CP^{n/2},\ \HP^{n/4},\ \CaP^2\mbox{ (with dimension $n=16$)}.
\end{align*}
We denote by $M_0\in\{S^n, \RP^n, \CP^{n/2}, \HP^{n/4}, \CaP^2\}$ the model of a Besse manifold $M$, i.e., the unique manifold $M_0$ from this list, satisfying $H^*(M;\Z)\cong H^*(M_0;\Z)$. All simply connected Besse manifolds $M$ are spin, except those whose associated model is $M_0=\CP^{n/2}$ with even complex dimension $n/2$; see~\cite{Radeschi:2017dz}.

Let $\Lambda M=W^{1,2}(S^1,M)$ be the free loop space of $M$, and let us identify $M\subset \Lambda M$ with the subspace of constant loops. The circle $S^1=\R/\Z$ acts on $\Lambda M$ by time-shift
\begin{align*}
 t\cdot\gamma=\gamma(t+\cdot),
 \qquad
 \forall t\in S^1,\ \gamma\in\Lambda M,
\end{align*}
and we can therefore consider the $S^1$-equivariant cohomology of $\Lambda M$.  Every non-zero cohomology class $\mu\in H^*_{S^1}(\Lambda M,M;\Q)$ gives rise to an associated spectral invariant $c_g(\mu)\in\sigma(SM)$, which is the period of a unit-speed closed geodesic of $(M,g)$. We will recall the precise definition of $c_g(\mu)$ in Subsection~\ref{ss:geodesic_flows_equivariant_spectral_invariants}.

Assume now that $M$ is an $n$-dimensional Zoll manifold, and consider a Zoll Riemannian metric on $M$. The associated geodesic flow gives rise to an $S^1$ action on the unit tangent bundle $SM$. In this case, it is well known that the $S^1$-equivariant cohomology of the free loop space relative to the constants admits an explicit isomorphism
\begin{align}
\label{e:rel_equiv_hom_loop_space}
H^*_{S^1}(\Lambda M,M;\Q)\cong\bigoplus_{m\geq1} H^{*-m\,i(M)-(m-1)(n-1)}(SM/S^1;\Q),
\end{align}
for some integer $i(M)>0$; see Section~\ref{ss:spectral_characterization_Riemannian}. We consider the non-zero cohomology classes
\begin{align*}
\alpha_m\in H^{m\,i(M)+(m-1)(n-1)}_{S^1}(\Lambda M,M;\Q),
\\
\beta_m\in H^{m\,i(M)+(m+1)(n-1)}_{S^1}(\Lambda M,M;\Q),
\end{align*}
which correspond to generators of $H^0_{S^1}(SM/S^1;\Q)$ and $H^{2n-2}_{S^1}(SM/S^1;\Q)$, respectively, of the $m$-th summand in the above direct sum decomposition. Note that the definition of these classes relies on the fact that $M$ admits a Zoll metric. (However, it is not immediately obvious that the classes $\alpha_m$ and $\beta_m$ are determined by $M$ and completely independent of the choice of the reference Zoll metric.) Our characterization of Zoll Riemannian metrics is as follows.

\begin{thm}
\label{t:geodesics_Zoll}
Let $M$ be a simply connected spin Zoll manifold of dimension $n\geq2$, and let $g$ be any Riemannian metric on $M$. Then, the following three conditions are equivalent:
\begin{itemize}
\item[(i)] $c_g(\alpha_1)=c_g(\beta_1)$, 
\item[(ii)] $c_g(\alpha_m)=c_g(\beta_m)=m\,c_g(\alpha_1)$ for all integers $m\geq1$,
\item[(iii)] $g$ is Zoll, and its unit speed geodesics have minimal period $c_g(\alpha_1)$.
\end{itemize}
Moreover, if $M=S^n$ with $n\neq3$, condition $(\mathrm{i})$ can be replaced by
\begin{itemize}
\item[(i')] $c_g(\alpha_m)=c_g(\beta_m)$ for some $m\in\N$.
\end{itemize}
\end{thm}

The moreover part of Theorem~\ref{t:geodesics_Zoll} is based on the validity of the so-called Berger conjecture, which was established for $S^2$ by Gromoll and Grove \cite{Gromoll:1981kl}, and recently for $S^n$ with $n>3$ by Radeschi and Wilking \cite{Radeschi:2017dz}: any Besse Riemannian metric on $S^n$ is Zoll. The conjecture is still open for the other simply connected closed manifolds admitting Zoll Riemannian metrics, and its validity would imply the equivalence of (i') and the Zoll condition in full generality.

\subsection{General perspective} All results stated in this section fit the same general pattern and can be cast in the framework of Lusternik--Schnirelmann theory. In its modern version the theory gives a lower bound on the number of critical values of a function  or a functional $f\colon X\to \R$ with isolated critical points in terms of the structure of the cup product on the cohomology of the underlying space $X$ or, more generally, the Morse-type homology of $f$. It hinges on the fact that for two cohomology classes $\alpha$ and $\beta$, we necessary have the inequality
\[
c(\alpha\smile\beta,f)\geq c(\alpha,f)
\] 
relating the minimax critical values of $f$ associated with the classes $\alpha\smile\beta$ and $\alpha$ and on the fact that this inequality is strict when $\beta$ has degree $|\beta|>0$ and the critical points are isolated.

Furthermore, when the equality occurs the critical set $K$ of $f$ at the critical level $c(\alpha\smile\beta,f)=c(\alpha,f)$ must be sufficiently large: the restriction $\beta|_K$ of $\beta$ to $K$ is necessarily non-zero. In particular, informally speaking, the dimension of $K$ must be at least $|\beta|$. These facts are well-known and in some form go back to the original work of Lusternik and Schnirelmann, cf.\ \cite{Viterbo:1997pi}. Moreover, they also hold in many other versions of Morse theory, e.g., in fixed point Floer theory; \cite{Howard:2012aa}.

This general principle carries over to the equivariant setting when $f$ is invariant under an $S^1$ action on $X$.  The multiplication by the image of the generator of $e\in H^2(BS^1;\R)$ gives rise to the shift operator $D$ on $H^*_{S^1}(X;\R)$ increasing degree by 2.  As above, $c(D\alpha,f)\geq c(\alpha,f)$ and the inequality is strict provided that the action is locally free and the critical sets are isolated $S^1$-orbits. Moreover, whenever the equality holds,  the operator $D$ is necessarily non-zero on $H^*_{S^1}(K;\R)$, where $K$ is again the critical set on the level $c(D\alpha,f)=c(\alpha,f)$. In particular, $c(D^k\alpha,f)=c(\alpha,f)$ implies that $K$ must have dimension at least $2k+1$, assuming again that the action is locally free. This setting is analyzed in detail in Section \ref{s:Fadell_Rabinowitz} where we discuss the Fadell--Rabinowitz index. 

Let us specialize these arguments to the case where $f$ is an action--type functional or the energy functional on a suitably defined loop space $X$ of a manifold $Y$ of dimension $2n+1$. If $c(D^n\alpha,f)=c(\alpha,f)$ for some class $\alpha$, we conclude that $Y$ is filled in by the critical points of $f$. The results of the previous sections are examples of this general principle.

The Lusternik--Schnirelmann type inequalities established in \cite{Ginzburg:2020aa} indicate that these results should also have analogues for equivariant Floer and symplectic  cohomology, readily leading to a systolic characterization of Reeb flows on more general contact manifolds. (However, yet a more general version of the results, mentioned above, with a homologically non-trivial critical set $K\subsetneq Y$ of closed Reeb orbits, identified with a subset of $Y$, of a smaller dimension encounters a technical difficulty. It stems from the fact that the Floer homology is not the homology of a loop space $X$, while $H^*_{S^1}(K;\R)$, unless $K$ is ``sufficiently nice", is defined in terms of the equivariant homology of a neighborhood of $K$ in the loop space $X$.)

Finally, note that on the conceptual level the results from \cite{Cristofaro-Gardiner:2019aa} also fit in this framework with the role of the equivariant cohomology taken by ECH and the $U$-operator in ECH playing the role of the shift operator $D$.

\subsection{Organization of the paper} In Section~\ref{s:Fadell_Rabinowitz} we provide the background on the Fadell--Rabinowitz index, and prove a technical result, Lemma~\ref{l:bound_FR}, which will be one of the key ingredients for the proofs of our theorems. In Section~\ref{s:convex} we prove the results concerning convex contact spheres. In Section~\ref{s:EH_capacities} we prove Theorem~\ref{t:EH_Besse} on restricted contact type hypersurfaces in $\R^{2n}$. Finally, in Section~\ref{s:geodesic_flows} we prove Theorem~\ref{t:geodesics_Zoll} concerning geodesic flows.

\subsection*{Acknowledgments} The authors are grateful to Barney Bramham for suggesting an improvement on the original statement of Theorem~\ref{t:Zoll_pinched} and to Jean Gutt and Andrzej Szulkin for useful discussions. Parts of this work were carried out while the second and third authors were in residence at MSRI, Berkeley, CA, during the Fall 2018 semester. The authors would like to thank the institute for its warm hospitality and support.

\section{The Fadell--Rabinowitz index}
\label{s:Fadell_Rabinowitz}
All equivariant spectral invariants used in this paper are based on a cohomological index introduced by Fadell--Rabinowitz  \cite{Fadell:1978aa}, which is a replacement for the cup-length in equivariant critical point theory. We denote by $e\in H^2(BS^1;\Q)$ the Euler class of the classifying vector bundle $ES^1\to BS^1$. Let $X$ be a topological space equipped with a continuous $S^1$ action. Denote by
$\pr_2:X\times_{S^1}ES^1\to BS^1$, $\pr_2([x,v])=[v]$ the natural projection map. The cohomology class 
\begin{align*}
e_X := \pr_2^* e\in H^2_{S^1}(X;\Q)
\end{align*} 
is the Euler class of the circle bundle 
\begin{align}
\label{e:pull_back_circle_bundle}
\pi:X\times ES^1\to X\times_{S^1} ES^1.
\end{align}
If $X\neq\varnothing$, the \textbf{Fadell--Rabinowitz index} $\fr(X)$ is the supremum of the integers $k\geq0$ such that 
\[e_X^k:=\underbrace{e_X\smile \dotsc \smile e_X}_{\times k}\neq0\mbox{ in }H^{2k}_{S^1}(X;\Q).\]
It is convenient to set $\fr(\varnothing):=-1$. Then the index becomes subadditive: if $A,B\subset X$ are two $S^1$-invariant open subsets, then 
\begin{align}
\label{e:subadditivity_FR}
\fr(A\cup B)\leq\fr(A)+\fr(B)+1. 
\end{align}
Moreover, it readily follows from its definition that the index is monotone with respect to the inclusion: $\fr(A)\leq\fr(B)$ if $A\subseteq B$.
If $X$ is the total space of a principal $S^1$-bundle $X\to B$ with Euler class $\tilde e$ (e.g., a prequantization bundle), then $\fr(X)$ is equal to the supremum of all $k\in \N$ such that $\tilde e^k\neq 0$ in $H^*(B;\Q)$.

The following technical lemma is an important ingredient in the proofs of our main theorems. We prove it for metric spaces, which is sufficient for our purposes. 

\begin{lem}
\label{l:bound_FR}
Let $X$ be a metric space equipped with a continuous $S^1$ action such that $\fr(X)<\infty$ and $K\subset X$ a compact $S^1$-invariant subset admitting arbitrarily small  neighborhoods $W\subseteq X$ with trivial cohomology $H^{2p+1}(W;\Q)$ for all $p>q$ for some integer $q>0$. Then  $K$ admits an $S^1$-invariant neighborhood $U\subseteq X$ with Fadell--Rabinowitz index $\fr(U)\leq q$.
\end{lem}

\begin{proof}
Since $\fr(X)<\infty$, the monotonicity of the Fadell--Rabinowitz index implies that every $S^1$-invariant neighborhood $V\subseteq X$ of $K$ satisfies $\fr(V)<\infty$ as well. Therefore,
\begin{align*}
 I := \big\{ \fr(V)\ \big|\ V\subseteq X \mbox{ is an $S^1$-invariant neighborhood of }K\big\}
\end{align*}
is a subset of the non-negative integers. In order to prove the lemma, we must show that $r:=\min I\leq q$.

We consider an $S^1$-invariant neighborhood $V\subseteq X$ of $K$ such that 
$\fr(V)=r$. Since $e_{V}^{r}\neq0$ and $e_{V}^{r+1}=0$, the Gysin exact sequence of the circle bundle~\eqref{e:pull_back_circle_bundle}, which reads
\begin{align*}
\dotsc
\ttoup^{\pi^*}  H^{2r+1}(V;\Q)
\ttoup^{\pi_*} H^{2r}_{S^1}(V;\Q)
\ttoup^{\smile e_{V}} H^{2r+2}_{S^1}(V;\Q)
\ttoup^{\pi^*} 
\dotsc,
\end{align*}
implies that there exists a non-zero $\mu\in H^{2r+1}(V;\Q)$ such that 
$\pi_*(\mu)=e_{V}^{r}$.

Let $W\subset V$ be a neighborhood of $K$ such that $H^{2p+1}(W;\Q)$ is trivial for all $p>q$. The assumptions that $X$ is a metric space and that $K$ is compact guarantee that there exists an $S^1$-invariant neighborhood $U\subseteq W$ of $K$. Notice that $\fr(U)=r$, since $r=\inf I$ and $\fr(U)\leq\fr(V)=r$. Let us consider the commutative diagram
\begin{align*}
  \xymatrix{
      H^{2r+1}(V;\Q) \ar[r]^{\iota_1^*} \ar[d]_{\pi_*} 
      & 
      H^{2r+1}(W;\Q) \ar[r]^{\iota_2^*}
      & 
      H^{2r+1}(U;\Q) \ar[d]^{\pi_*} 
      \\ 
      H^{2r}_{S^1}(V;\Q) \ar[rr]^{\iota_3^*} 
      & 
      & 
      H^{2r}_{S^1}(U;\Q)      
 }   
\end{align*}
Here, the horizontal homomorphisms are induced by the inclusions $W\subseteq V$, $U\subseteq W$, and $U\subseteq V$. Notice that 
\begin{align*}
0\neq e_{U}^{r}
=
\iota_3^*(e_V^{r})
=
\iota_3^*\circ\pi_*(\mu)
=
\pi_*\circ\iota_2^*\circ\iota_1^*(\mu).
\end{align*}
In particular, $H^{2r+1}(W;\Q)\neq0$, and therefore $r\leq q$.
\end{proof}

\section{Convex contact spheres}
\label{s:convex}

\subsection{Closed Reeb orbits on convex contact spheres}
\label{ss:closed_Reeb_orbits_convex_spheres}
Let $\Sigma\subset\R^{2n}$ be a convex contact sphere, $\lambda$ its canonical contact form, and $R$ its Reeb vector field. We are interested in the closed Reeb orbits of $(\Sigma,\lambda)$, i.e., the curves $\gamma:\R\to\Sigma$ which are solutions of
\begin{align}
\label{e:closed_Reeb}
\left\{
  \begin{array}{@{}l}
    \dot\gamma(t)=R(\gamma(t)), \\ 
    \gamma(0)=\gamma(\tau)\mbox{ for some minimal }\tau=\tau_\gamma>0. 
  \end{array}
\right.
\end{align}
The action spectrum of $\Sigma$ is the set of positive numbers
\begin{align*}
\sigma(\Sigma) := \big\{k\tau_\gamma\ \big|\ k\in\N,\ \gamma\mbox{ solution of \eqref{e:closed_Reeb}}\big\},
\end{align*}
where $\N=\{1,2,3, \dotsc\}$ is the set of natural numbers.

The solutions of~\eqref{e:closed_Reeb} can be studied by means of Clarke's variational principle, which we will briefly recall in the next subsection following Ekeland's monograph \cite{Ekeland:1990lc}; we refer the reader to \cite{Ekeland:1984aa, Ekeland:1987aa} for alternative, although analogous, approaches. A preliminary step consists in converting the problem~\eqref{e:closed_Reeb} into an analogous Hamiltonian periodic orbit problem with prescribed period. To this end, we fix once and for all $\alpha\in(1,2)$, and consider the Hamiltonian $H:\R^{2n}\to[0,\infty)$ which is positively homogeneous of degree $\alpha$ and such that $\Sigma=H^{-1}(1)$. Since $\Sigma$ is strongly convex and encloses the origin, such a Hamiltonian exists, is unique and smooth with positive-definite Hessian outside the origin. The non-trivial 1-periodic Hamiltonian orbits $\zeta:\R\to\R^{2n}$ of $H$ are the solutions of
\begin{align}
\label{e:closed_Hamiltonian}
\left\{
  \begin{array}{@{}l}
    \dot\zeta(t)=J\nabla H(\zeta(t)), \\ 
    \zeta(0)=\zeta(1)\neq0.
  \end{array}
\right.
\end{align}
There is a one-to-one correspondence between solutions $\gamma$ of~\eqref{e:closed_Reeb} and countable sequences of solutions $\{\gamma_k\ |\ k\in\N\}$ of \eqref{e:closed_Hamiltonian}, given by 
\begin{align}
\label{e:Hamiltonian_Reeb}
\gamma_k(t)=\big(\tfrac{2k\tau_\gamma}{\alpha}\big)^{1/(\alpha-2)} \gamma(k\tau_\gamma t).
\end{align}
In other words, $\gamma_k$ is a rescaling of the $k$-th iterate of $\gamma$.

\subsection{Variational setting}
Consider the exponent $\beta=\alpha/(\alpha-1)\in(2,\infty)$ which is H\"older conjugate to $\alpha$. The Legendre dual of $H$ is the function
\begin{align*}
H^*:\R^{2n}\to\R,\qquad
H^*(w)=\max_{z\in\R^{2n}}\Big( \langle z,w\rangle - H(z) \Big),
\end{align*}
which is positively homogeneous of degree $\beta$, smooth with positive-definite Hessian outside the origin, and $C^2$ at the origin. Setting $S^1=\R/\Z$, denote by $L^\beta_0(S^1,\R^{2n})$ the space of functions in $L^\beta(S^1,\R^{2n})$ with zero average. Namely, every such function is of the form $\dot\zeta$ for some $\zeta\in W^{1,\beta}(S^1,\R^{2n})$. The Clarke action functional is defined by
\begin{align*}
\Psi: L^\beta_0(S^1,\R^{2n})\to\R,
\qquad
\Psi(\dot\zeta) = \int_0^1 \Big(-\tfrac12\langle J\zeta(t),\dot\zeta(t)\rangle + H^*(-J\dot\zeta(t)) \Big)\diff t,
\end{align*}
The circle $S^1$ acts on the Banach space $L^\beta_0(S^1,\R^{2n})$ by the time-shift
\begin{align*}
 t\cdot \dot\zeta = \dot\zeta(t+\cdot),\qquad\forall t\in S^1,\ \dot\zeta\in L^\beta_0(S^1,\R^{2n}),
\end{align*}
and $\Psi$ is invariant under this action.

The Clarke action functional $\Psi$ satisfies all standard conditions required to apply the classical variational methods: it is $C^{1,1}$, bounded from below, and satisfies the Palais--Smale condition. Beside the origin, its critical points are precisely $\dot\zeta\in C^\infty(S^1;\R^{2n})$ admitting a primitive $\zeta$ which is a 1-periodic solution of the Hamiltonian system~\eqref{e:closed_Hamiltonian}. Therefore, in the notation from~\eqref{e:Hamiltonian_Reeb}, there is a one-to-one correspondence between solutions $\gamma$ of~\eqref{e:closed_Reeb} and sequences of critical circles 
\[ \bigcup_{k\in\N} S^1\cdot\dot\gamma_k\subset\crit(\Psi)\]
with associated critical values
\begin{align*}
\Psi(\dot\gamma_k) 
= - \big(1- \tfrac\alpha2 \big) H(\gamma_k(t))
= - \big(1- \tfrac\alpha2 \big) (2k\alpha^{-1}\tau_\gamma)^{-\alpha/(2-\alpha)}<0.
\end{align*}
The origin is thus the only critical point of $\Psi$ with non-negative critical value, and indeed $\Psi(0)=0$.
It is notationally convenient to renormalize $\Psi$ by introducing
\begin{align*}
\A:\{\Psi<0\}\to(0,\infty),
\qquad
\A(\dot\zeta) = \tfrac\alpha2 \big( \big(\tfrac{2}{\alpha-2}\big) \Psi(\dot\zeta) \big)^{\frac{\alpha-2}{\alpha}},
\end{align*}
so that $\crit(\A)=\crit(\Psi)\setminus\{\bm0\}$, and the critical value of each $\gamma_k$ is precisely the action of the $k$-th iterate of the associated closed Reeb orbit $\gamma$, i.e.
\begin{align*}
\A(\dot\gamma_k) = k\tau_\gamma \in \sigma(\Sigma).
\end{align*}

\subsection{Equivariant spectral invariants}
\label{ss:convex_spectral_invariants}

Let us recall the construction of the equivariant spectral invariants $\{c_i(\Sigma)\ |\ i\geq0\}$ for a convex contact sphere $\Sigma$, which is originally due to Ekeland and Hofer \cite{Ekeland:1987aa}. In the setting of the Clarke action functional \cite[Section~V.3]{Ekeland:1990lc}, such a construction is based on the fact that $\fr(\{\Psi<0\})=\infty$. Since $\Psi$ satisfies the Palais--Smale condition, for each integer $i\geq0$ the real number 
\begin{align*}
\tilde c_i(\Sigma):= \inf\big\{a\in\R\ \big|\ \fr(\{\Psi<a\})\geq i\big\}
\end{align*}
is a negative critical value of the Clarke action functional $\Psi$, and $\tilde c_i(\Sigma)\to0$ as $i\to\infty$.  Namely, the value
\begin{align*}
c_i(\Sigma)
:=
\tfrac\alpha2 \big( \big(\tfrac{2}{\alpha-2}\big) \tilde c_i(\Sigma) \big)^{(\alpha-2)/\alpha}
=
\inf\big\{a\in\R\ \big|\ \fr(\{\A<a\})\geq i\big\} 
\end{align*}
belongs to the action spectrum $\sigma(\Sigma)$, and
\begin{align}
\label{e:growth_spectral_invariants}
\lim_{i\to\infty} c_i(\Sigma)=\infty.
\end{align}

\subsection{Cohomology of neighborhoods of critical sets}

We denote by $\phi_R^t:\Sigma\to\Sigma$ the Reeb flow on $\Sigma$, i.e.\ $\phi_R^0=\id$ and $\tfrac{\diff}{\diff t}\phi_R^t=R\circ\phi_R^t$. We fix $\tau>0$, and consider the (possibly empty) compact subset
\begin{align*}
K:=\fix(\phi_R^\tau)\subseteq\Sigma.
\end{align*}
Namely, for each $z\in K$, the curve $\gamma_z(t):=\phi_R^t(z)$ is a $\tau$-periodic Reeb orbit. As follows from Equation~\eqref{e:Hamiltonian_Reeb}, $\gamma_z$ has an associated critical point $\dot\zeta_z\in\crit(\A)$ given by
\begin{align}
\label{e:zeta_z}
\zeta_z(t):=
(2\tau \alpha^{-1})^{1/(\alpha-2)} \gamma_z(\tau t),
\end{align}
and with critical value $\A(\dot\zeta_z)=\tau$. We set
\begin{align*}
\KK:=\big\{ \dot\zeta_z\in L^\beta_0(S^1,\R^{2n})\ \big|\ z\in K \big\}.
\end{align*}

\begin{lem}
\label{l:neighborhoods}
We have $\fr(\UU)\geq n-1$ for all sufficiently small $S^1$-invariant neighborhoods $\UU\subset L^\beta_0(S^1,\R^{2n})$ of $\KK$ if and only if $K=\Sigma$.
\end{lem}

\begin{proof}
We define the smooth map
\begin{align*}
\iota: W^{1,\beta}(\R/\tau\Z,\Sigma)\to L^\beta_0(S^1,\R^{2n})
\end{align*}
by $\iota(\gamma)= \dot\zeta$, where 
$\zeta(t)=(2\tau \alpha^{-1})^{1/(\alpha-2)} \gamma(\tau t)$. In particular, using the notation from Equation~\eqref{e:zeta_z}, $\iota(\gamma_z)=\dot\zeta_z$.

It is well known that the Reeb vector fields are geodesible. More specifically, if we define $g$ to be any Riemannian metric on $\Sigma$ such that $g(R,\cdot)=\lambda$, then the orbits of the Reeb flow $\phi_R^t:\Sigma\to\Sigma$ are unit-speed geodesics of the Riemannian manifolds $(\Sigma,g)$. We fix one such $g$ from now on.
Let $\epsilon>0$ be a small enough quantity such that, for each $z\in \Sigma$ and $t\in[0,\epsilon]$, we have $d_g(z,\phi_R^t(z))<\injrad(\Sigma,g)$. Here, $d_g$ and $\injrad(\Sigma,g)$ denote the Riemannian distance and the injectivity radius of $(\Sigma,g)$ respectively. We choose an arbitrarily small open neighborhood $W\subseteq\Sigma$ of the compact subset $K$, which is small enough so that 
\begin{align*}
d_g(\phi_R^{\tau-\epsilon}(z),z)<\injrad(\Sigma,g),
\qquad
\forall z\in W.
\end{align*}
For each $z\in W$, we define $\tilde\gamma_z\in W^{1,\beta}(\R/\tau\Z,\Sigma)$ to be the curve such that $\tilde\gamma_z(t)=\phi_R^t(z)$ for all $t\in[0,\tau-\epsilon]$, and $\gamma_z|_{[\tau-\epsilon,\tau]}$ is the unique shortest geodesic of $(\Sigma,g)$ parametrized with constant speed and joining $\phi_R^{\tau-\epsilon}(z)$ and $z$. Notice that, since the Reeb orbits of $\Sigma$ are unit-speed geodesics of $(\Sigma,g)$, we have 
\[\tilde\gamma_z(t)=\gamma_z(t)=\phi_R^t(z),\qquad \forall z\in K,\ t\in\R/\tau\Z.\] 
After shrinking $W$ if necessary, we see that the map 
\[ 
\tilde\gamma:W\hookrightarrow W^{1,\beta}(\R/\tau\Z,\Sigma),
\qquad
\tilde\gamma(z)=\tilde\gamma_z
\] 
is a smooth embedding. The space
$\WW := \iota\circ\tilde\gamma(W)$
is thus a $(2n-1)$-dimensional smooth submanifold of $L^\beta_0(S^1,\R^{2n})$ diffeomorphic to the  open subset $W\subseteq\Sigma$. We denote by 
\begin{align}
\label{e:tubular_nbhd}
\NN\subset L^\beta_0(S^1,\R^{2n})
\end{align}
a tubular neighborhood of $\WW$, which is thus homotopy equivalent to $W$. In particular,
\begin{align*}
H^{2n-1}(\NN;\Q)\cong
H^{2n-1}(\WW;\Q)\cong
H^{2n-1}(W;\Q).
\end{align*}
Both $\WW$ and $\NN$ can be chosen arbitrarily small, and thus this construction provides a fundamental system of open neighborhoods of the critical set $\KK$.

If $K\neq \Sigma$, then $K$ admits an arbitrarily small open neighborhood $W\subsetneq\Sigma$, and in particular $H^{*\geq 2n-1}(W;\Q)=0$. Therefore, by the previous paragraph, $\KK$ admits an arbitrarily small neighborhood $\NN\subset L^\beta_0(S^1,\R^{2n})$ with $H^{*\geq 2n-1}(\NN;\Q)=0$. We can thus apply the abstract Lemma~\ref{l:bound_FR}, and infer that $\KK$ admits an $S^1$-invariant neighborhood $\UU\subset L^\beta_0(S^1,\R^{2n})$ with $\fr(\UU)<n-1$.

Finally, if $K=\Sigma$, every point $z\in\Sigma$ lies on a closed Reeb orbit of period $\tau$. Therefore, for each $z\in\Sigma$, the curve 
$\dot\zeta_z\in L^\beta_0(S^1,\R^{2n})$ is a critical point of $\A$ with critical value $\A(\dot\zeta_z)=\tau$. The time-rescaled Reeb flow $t\mapsto\phi_R^{\tau t}$ defines a locally-free $S^1$ action on $\Sigma$, and the homeomorphism 
\[\Sigma\to \KK,\qquad z\mapsto \dot\zeta_z\] 
is $S^1$-equivariant. This implies that
$\fr(\KK) = \fr(\Sigma)$. The Gysin sequence
\begin{align*}
...\ttoup^{\pi^*} H^{*+1}(\Sigma;\Q) \ttoup^{\pi_*} H_{S^1}^*(\Sigma;\Q) \ttoup^{\smile e_\Sigma} H_{S^1}^{*+2}(\Sigma;\Q) \ttoup^{\pi^*} H^{*+2}(\Sigma;\Q)\ttoup^{\pi_*} ...
\end{align*}
of the circle bundle $\pi:\Sigma\times ES^1\to \Sigma\times_{S^1} ES^1$ readily implies that $e_\Sigma^{n-1}\neq0$ in $H^{2n-2}_{S^1}(\Sigma;\Q)$, and thus \[\fr(\KK) = \fr(\Sigma)\geq n-1.\] 
 We conclude that every $S^1$-invariant neighborhood $\UU\subset L^\beta_0(S^1,\R^{2n})$ has Fadell--Rabinowitz index $\fr(\UU)\geq \fr(\KK)\geq n-1$.
\end{proof}

\subsection{Spectral characterization of the Besse and Zoll conditions}

We can now prove one implication in Theorem~\ref{t:Besse_convex}.

\begin{lem}
\label{l:Besse_convex_1}
Let $\Sigma\subset\R^{2n}$ be a convex contact sphere. If $c_i(\Sigma)=c_{i+n-1}(\Sigma)$ for some $i\geq0$, then $\Sigma$ is Besse and every Reeb orbit has (not necessarily minimal) period $c_i(\Sigma)$.
\end{lem}

\begin{proof}
We recall that $c_i(\Sigma)$ is obtained as the minmax of $\A$ associated to the $i$-th power of the Euler class of the circle bundle \eqref{e:pull_back_circle_bundle} with $X=L^\beta_0(S^1,\R^{2n})$.
Therefore, if $c_i(\Sigma)=c_{i+n-1}(\Sigma)$,  the classical Lusternik--Schnirelmann theorem \cite[Theorem~1.1]{Viterbo:1997pi} implies that every $S^1$-invariant neighborhood $\UU\subset L^\beta_0(S^1,\R^{2n})$ of $\crit(\A)\cap\A^{-1}(c_i(\Sigma))$ has Fadell--Rabinowitz index $\fr(\UU)\geq n-1$. By Lemma~\ref{l:neighborhoods}, every point $z\in\Sigma$ must lie on a closed Reeb orbit of period $c_i(\Sigma)$.
\end{proof}

We postpone the proof of the other implication in Theorem~\ref{t:Besse_convex}, and we first deal with the remaining statements concerning convex contact spheres.

\begin{proof}[Proof of Corollary~\ref{c:Zoll_convex}]
If $c_0(\Sigma)=c_{n-1}(\Sigma)$, then Lemma~\ref{l:Besse_convex_1} implies that $\Sigma$ is Besse, and every Reeb orbit has period $c_0(\Sigma)$. Since $c_0(\Sigma)=\min\A=\sys(\Sigma)$, it must actually be the minimal period of each Reeb orbit. Hence $\Sigma$ is Zoll.

The opposite implication is well known and analogous to the last paragraph in the proof of Lemma~\ref{l:neighborhoods}. Indeed, if $\Sigma$ is Zoll, then for each $z\in\Sigma$ the curve 
\[
\dot\zeta_z\in L^\beta_0(S^1,\R^{2n}), 
\qquad
\zeta_z(t)=(2c_0(\Sigma) \alpha^{-1})^{1/(\alpha-2)} \phi_R^{c_0(\Sigma) t}(z)
\]
is a global minimizer of $\A$, i.e.,
\[\A(\dot\zeta_z)=\min\A=c_0(\Sigma).\] 
The time-rescaled Reeb flow $t\mapsto\phi_R^{c_0(\Sigma) t}$ defines a free $S^1$ action on $\Sigma$, and the homeomorphism 
\[\Sigma\to \A^{-1}(c_0(\Sigma)),\qquad z\mapsto\dot\zeta_z\] 
is $S^1$-equivariant. Therefore,
\begin{align*}
\fr(\A^{-1}(c_0(\Sigma))) = \fr(\Sigma) = n-1,
\end{align*}
which readily implies $c_{n-1}(\Sigma)=c_0(\Sigma)$.
\end{proof}

\begin{proof}[Proof of Theorem~\ref{t:Zoll_pinched}]
If $\Sigma$ is Zoll, then $\sigma(\Sigma)=\{k\,\sys(\Sigma)\ |\ k\in\N\}$ and thus the intersection $\sigma(\Sigma)\cap(\sys(\Sigma),2\,\sys(\Sigma))$ is empty. Conversely, for $\delta\in(1,\sqrt2]$, assume that the action spectrum of a convex $\delta$-pinched contact sphere $\Sigma\subset\R^{2n}$ satisfies $\sigma(\Sigma)\cap(\sys(\Sigma),\delta^2\sys(\Sigma))=\varnothing$. We now argue building on the work of Ekeland and Lasry \cite{Ekeland:1980aa}, and we will follow the exposition in Ekeland's monograph \cite[Section~V.2]{Ekeland:1990lc}. We choose $0<r<R$ such that $R/r<\delta$ and the filling $C\subset\R^{2n}$ of $\Sigma$ satisfies $B^{2n}(r)\subseteq C\subseteq B^{2n}(R)$. We consider the Hamiltonian $H:\R^{2n}\to[0,\infty)$ which is positively homogeneous of degree $\alpha\in(1,2)$ and satisfies $H^{-1}(1)=\Sigma$, its Legendre dual $H^*:\R^{2n}\to[0,\infty)$ which is positively homogeneous of degree $\beta=\alpha/(\alpha-1)$, and the associated renormalized Clarke action functional $\A$. The function $H^*$ is squeezed between the corresponding dual functions associated to the round spheres $\partial B^{2n}(r)$ and $\partial B^{2n}(R)$, i.e.,
\begin{align}
\label{e:sandwich_H*}
\beta^{-1} (r^{\alpha}/\alpha)^{\beta-1} \|u\|^\beta
\leq
H^*(u)
\leq
\beta^{-1} (R^{\alpha}/\alpha)^{\beta-1} \|u\|^\beta.
\end{align}
The left-most inequality readily implies 
\begin{align*}
\pi r^2\leq \sys(\Sigma).
\end{align*}
Consider now the unit sphere $S^{2n-1}\subset\R^{2n}$ equipped with the usual Hopf circle action
$t\cdot z= e^{2\pi Jt}z$ for all $t\in S^1$ and $z\in S^{2n-1}$.
The smooth embedding
\begin{align*}
\iota:S^{2n-1}\hookrightarrow L^\beta_0(S^1,\R^{2n}),
\qquad
\iota(z)(t)= (2\pi)^{\frac{1-\alpha}{2-\alpha}} (\alpha/R^\alpha)^{\frac{1}{2-\alpha}} e^{2\pi Jt} z
\end{align*}
is $S^1$-equivariant, and the right-most inequality in~\eqref{e:sandwich_H*} implies 
\begin{align*}
\max\A\circ\iota\leq\pi R^2.
\end{align*}
Since $\fr(S^{2n-1})=n-1$, we readily infer that 
\begin{align*}
\fr(\{\A\leq\pi R^2\})\geq \fr(\iota(S^{2n-1}))=\fr(S^{2n-1})=n-1,
\end{align*}
and therefore
\begin{align}
\label{e:bound_c_n-1}
c_{n-1}(\Sigma)
\leq
\pi R^2
<
\delta^2 \pi r^2
\leq
\delta^2\sys(\Sigma)
.
\end{align}
Since $c_0(\Sigma)\leq c_{n-1}(\Sigma)$ and $\sigma(\Sigma)\cap(\sys(\Sigma),\delta^2\sys(\Sigma))=\varnothing$, the inequality \eqref{e:bound_c_n-1} implies that $c_{n-1}(\Sigma)=\sys(\Sigma)=c_0(\Sigma)$, and Corollary~\ref{c:Zoll_convex} implies that $\Sigma$ is Zoll.
\end{proof}

\subsection{Morse indices of the Clarke action functional}\label{ss:Morse_indices_Clarke}
 
The fact that the Clarke action functional $\Psi$ (and its renormalized version $\A$) is not $C^2$ does not pose any problem for its Morse theory. Actually, Ekeland and Hofer, \cite{Ekeland:1987aa}, showed that, after applying a suitable saddle point reduction, $\Psi$ becomes $C^2$. Without entering into the technical details of this procedure, let us quickly recall the definition and properties of the Morse indices of $\Psi$. Let $\gamma:\R/\tau\Z\to\Sigma$ be a closed Reeb orbit and $\dot\zeta$ be the associated critical point of $\A$ given by
\[\zeta(t)=(2\tau \alpha^{-1})^{1/(\alpha-2)} \gamma(\tau t)\]
so that $\A(\dot\zeta)=\tau$.
Consider the quadratic function $Q:L^2_0(S^1;\R^{2n})\to\R,$
\begin{align*}
Q(\dot\eta)=\int_0^1 \Big( \langle J\eta(t),\dot\eta(t)\rangle + \nabla^2 H^*(-J\dot\zeta(t))[\dot\eta(t),\dot\eta(t)] \Big)\,\diff t.
\end{align*}
We define the Morse index $\ind(\dot\zeta)$ and the nullity $\nul(\dot\zeta)$, respectively, as the dimension of the negative eigenspace and of the kernel of the bounded self-adjoint operator $P$ on $L^2_0(S^1;\R^{2n})$ associated to the quadratic form $Q$, i.e.\ $Q(\dot\eta)=\langle P\dot\eta,\dot\eta\rangle_{L^2}$. Even though these are not exactly the standard definitions, it turns out that both indices are finite and the ordinary results from Morse theory apply to $\A$ with these two indices; see \cite[Section~IV.3]{Ekeland:1990lc}.

We denote by $\phi_H^t:\R^{2n}\to\R^{2n}$ the Hamiltonian flow of $H$, i.e.\ $\phi_H^0=\id$ and $\tfrac{\diff}{\diff t}\phi_H^t=J\nabla H\circ\phi_H^t$. Recall that $\zeta(t)=\phi_H^t(\zeta(0))$ is the 1-periodic Hamiltonian orbit associated to $\gamma$. On the energy hypersurface $\Sigma=H^{-1}(1)$, the $\tau$-periodic Reeb orbit $\gamma$ can be reparametrized into a $\tfrac2\alpha\tau$-periodic Hamiltonian orbit, for $\phi_H^t(\gamma(0))=\gamma(\tfrac\alpha2t)$.
We consider the continuous path of symplectic matrices 
\[
\Gamma_\alpha:[0,1]\to\Sp(2n),
\qquad
\Gamma_\alpha(t)=\diff\phi_H^{2\alpha^{-1}\tau t}(\gamma(0)).
\]
Here we write $\Gamma_\alpha$ with a subscript $\alpha$ in order to stress its dependence from $\alpha$, for the Hamiltonian $H$ is the positively homogeneous function of degree $\alpha$ such that $H^{-1}(1)=\Sigma$. The path $\Gamma_\alpha$ has a well defined \textbf{Conley--Zehnder index} $\cz(\Gamma_\alpha)$, which can be defined as follows. We denote by $K_n:=\big\{M\in\Sp(2n)\ \big|\ \det(M-I)=0\big\}$ the Maslov cycle. The symplectic group has fundamental group $\pi_1(\Sp(2n))\cong\Z$, and the complement $\Sp(2n)\setminus K_n$ has two connected components. If $\Gamma_\alpha$ is a ``non-degenerate'' path, i.e., $\Gamma_\alpha(1)\not\in K_n$ and $\Gamma_\alpha|_{(0,1)}$ intersects $K_n$ only outside its singular locus and transversely, then $\cz(\Gamma_\alpha)$ is the suitably oriented intersection number of $\Gamma_\alpha|_{(0,1)}$ and $K_n$. The index of a degenerate path is then defined in such a way that the function $\Psi\mapsto\cz(\Psi)$ is the point-wise largest lower semicontinuous extension of the index function on the space of non-degenerate paths. It was proved by Brousseau \cite{Brousseau:1990aa} (see also \cite[Lemmas~1.3-4]{Long:1998aa}) that the Morse indices of the Clarke's action functional are related to the indices of $\Gamma_\alpha$ as
\begin{align}
\label{e:Morse_CZ}
\ind(\dot\zeta)=\cz(\Gamma_\alpha)-n,
\qquad
\nul(\dot\zeta)=\dim\ker(\Gamma_\alpha(1)-I).
\end{align}
By \cite[Proposition~I.7.5]{Ekeland:1990lc}, the Conley--Zehnder index $\cz(\Gamma_\alpha)$ is actually independent of $\alpha\in(1,2]$, i.e.
\begin{align*}
 \cz(\Gamma_\alpha)=\cz(\Gamma_2),\qquad\forall\alpha\in(1,2).
\end{align*}
Notice that up to now we have only considered $\alpha\in(1,2)$, but the symplectic path $\Gamma_\alpha$ is perfectly defined for $\alpha=2$ as well. Moreover, according to a result due to Ekeland \cite[Theorem~I.4.6]{Ekeland:1990lc}, the Morse index can be computed by suitably counting the multiplicity of the conjugate points along the symplectic path $\Gamma_\alpha$ as
\begin{align}
\label{e:Morse_index_formula}
\ind(\dot\zeta)=\sum_{t\in(0,1)} \dim\ker(\Gamma_\alpha(t)-I).
\end{align}

Since the Hamiltonian $H$ is positively homogeneous of degree $\alpha\in(1,2)$, the map $\Gamma_\alpha(1)$ preserves the symplectic vector subspace $E:=\textrm{span}\{\dot\gamma(0),\gamma(0)\}$, and indeed
\begin{align*}
\Gamma_\alpha(1)\dot\gamma(0)=\dot\gamma(0),
\qquad
\Gamma_\alpha(1)\gamma(0)=(\alpha-2)\tau\dot\gamma(0)+\gamma(0).
\end{align*}
We denote as usual by 
\[E^\omega=\big\{v\in\R^{2n}\ \big|\ \langle v,Jw\rangle=0\quad \forall w\in E\big\}\] 
the symplectic orthogonal to $E$. Notice that $E^\omega\subset\Tan_{\gamma(0)}H^{-1}(1)$, and if we consider the decomposition $\R^{2n}=E\oplus E^\omega$, the symplectic matrix $\Gamma_\alpha(1)$ can be written in symplectic blocks as
\begin{align}
\label{e:Gamma_alpha_1}
\Gamma_\alpha(1)
=
\left(
  \begin{array}{@{}cc@{}}
    M_\alpha & 0 \\ 
    0 & N 
  \end{array}
\right).
\end{align}
The first block is given by
\begin{align}
\label{e:M_alpha}
M_\alpha=
\left(
  \begin{array}{@{}cc@{}}
    1 & (\alpha-2)\tau \\ 
    0 & 1
  \end{array}
\right),
\end{align}
whereas the second block $N=\diff\phi_R^{\tau}(\gamma(0))|_{E^\omega}$ is independent of $\alpha$; see \cite[pages~69--72]{Ekeland:1990lc}.

If $\Sigma$ is Besse and $\KK\subset\crit(\A)$ is a connected component, then the Morse index $\ind(\dot\zeta)$ and the nullity $\nul(\dot\zeta)$ are the same for all $\dot\zeta\in\KK$, and therefore we will simply write them as $\ind(\KK)$ and $\nul(\KK)$ respectively.
We will need the following remark concerning the Morse index. In the context of geodesic flows, an analogous statement was proved by Wilking \cite[Theorem~3]{Wilking:2001pi}.

\begin{lem}
\label{l:Besse_indices}
If $\Sigma$ is Besse, then every connected component $\KK\subset\crit(\A)$ is an odd dimensional closed manifold with nullity $\nul(\KK)=\dim(\KK)$ and even Morse index $\ind(\KK)$.
\end{lem}

\begin{proof}
Assume that $\Sigma$ is Besse, and let $\KK\subset\crit(\A)$ be a connected component of the critical set. We employ the notation of the previous paragraphs with respect to an arbitrary critical point $\zeta\in\KK$. It is well known that $\KK$ is an odd dimensional closed manifold, and if we consider the symplectic matrix $N$ in~\eqref{e:Gamma_alpha_1} we have $\dim(\KK)=1+\dim\ker(N-I)$; see \cite[Section~4.1]{Cristofaro-Gardiner:2019aa}. Since $\alpha\in(1,2)$, we conclude that 
\[\nul(\dot\zeta)=\dim\ker(\Gamma_\alpha(1)-1)=1+\dim\ker(N-I)=\dim(\KK).\]

As for the index, let us first remark that the parity of the Conley--Zehnder index of a symplectic path $\Psi:[0,1]\to\Sp(2n)$ with $\Psi(0)=I$ only depends on the final point $\Psi(1)$. Indeed, if $\Upsilon:[0,1]\to\Sp(2n)$ is another continuous path such that $\Upsilon(0)=I$ and $\Upsilon(1)=\Psi(1)$, then $\Upsilon$ is homotopic (with fixed endpoints) to the concatenation of $\Psi$ with a certain number of positive and negative full turns in $\Sp(2n)$, and every such full turn contributes to the index with a summand $\pm2$. Therefore, if $M:=\Psi(1)\in\Sp(2n)$, we can define
\begin{align*}
\parity(M):=\cz(\Psi)\bmod 2 \in\{0,1\}.
\end{align*}

We already mentioned that $\cz(\Gamma_\alpha)=\cz(\Gamma_2)$. We readily see from~\eqref{e:M_alpha} that $M_2=I_2\in\Sp(2)$.
By Wadsley's theorem \cite{Wadsley:1975sp}, there exists an integer $k\geq1$ such that $k\tau$ is a common period for the closed Reeb orbits. Therefore, in~\eqref{e:Gamma_alpha_1} the symplectic matrix $N$ must satisfy $N^k=I_{2n-2}\in\Sp(2n-2)$. In particular, we can find a symplectic matrix $P\in\Sp(2n)$ such that  $PNP^{-1}$ decomposes in symplectic blocks
\begin{align*}
PNP^{-1}
=
\left(
  \begin{array}{cc}
    N_1 & 0 \\ 
    0 & N_2 \\ 
  \end{array}
\right)
\end{align*}
where $N_1=I_{2n_1}\in\Sp(2n_1)$, and $N_2\in\Sp(2n_2)$ is a symplectic matrix whose eigenvalues are all contained in $S^1\setminus\{1\}$. Here, $S^1$ denotes the unit circle in the complex plane. Since the Conley--Zehnder index is invariant under symplectic conjugation and is additive on symplectic blocks, we have
\begin{align*}
\parity(N) = (\parity(N_1) + \parity(N_2))\bmod 2.
\end{align*}
Since $\sigma(N_2)\subset S^1\setminus\{1\}$, the symplectic matrix $N_2$ is contained in the same connected components of $\Sp(2n_2)\setminus K_{n_2}$ as $-I_{2n_2}$, and therefore
\begin{align*}
\parity(N_2)=\parity(-I_{2n_2}).
\end{align*}
It is well known that the symplectic path $\Psi_1:[0,1]\to\Sp(2n_1)$, $\Psi_1(t)=e^{2\pi Jt}$ has Conley--Zehnder index $\cz(\Psi_1)=n_1$, and analogously $\Psi_2:[0,1]\to\Sp(2n_2)$, $\Psi_2(t)=e^{\pi Jt}$ has Conley--Zehnder index $\cz(\Psi_2)=n_2$. Therefore, 
\begin{align*}
\parity(N_i)=\cz(\Psi_i)\bmod 2 = n_i \bmod 2,\qquad i=1,2,
\end{align*}
and analogously $\parity(M_2)=1$. By~\eqref{e:Morse_CZ}, we conclude that
\begin{align*}
\ind(\dot\zeta)\bmod 2
& =(\cz(\Gamma_\alpha)-n)\bmod 2\\
& =(\cz(\Gamma_2)-n) \bmod 2 \\
& =(\parity(M_2) + \parity(N_1) + \parity(N_2) -n)\bmod 2\\
& =(1+n_1+n_2-n) \bmod 2 \\
& =0.
\qedhere 
\end{align*}
\end{proof}

We can finally prove the remaining implication in Theorem~\ref{t:Besse_convex}. We recall that a convex contact sphere $\Sigma$ is called $(\tau,\mu)$-Besse when every Reeb orbit is $\tau$-periodic and has Conley-Zehnder index $\mu$ at period $\tau$. For notational convenience, we set $c_{-1}(\Sigma):=0$.

\begin{lem}\label{l:Besse_convex_2}
If $\Sigma\subset\R^{2n}$ is a $(\tau,\mu)$-Besse convex contact sphere, then $i:=(\mu-n)/2$ is a non-negative integer and
\[c_{i-1}(\Sigma)<\tau= c_i(\Sigma)=c_{i+n-1}(\Sigma)<c_{i+n}(\Sigma).\]
\end{lem}

\begin{proof}
The statement follows from Lemma~\ref{l:Besse_indices} and the lacunarity principle in Morse theory. Indeed, consider the critical manifold $\KK:=\crit(\A)\cap\A^{-1}(\tau)$. Formula~\eqref{e:Morse_index_formula} for the Morse index implies that for all connected components $\KK'\subset\crit(\A)$ with $\A(\KK')<\tau$ (if they exist at all) we have 
$\ind(\KK')+\nul(\KK')\leq\ind(\KK)$.
By Lemma~\ref{l:Besse_indices}, $\ind(\KK')+\nul(\KK')$ is odd, whereas $\ind(\KK)$ is even. Therefore
\begin{align*}
\ind(\KK')+\nul(\KK')<\ind(\KK).
\end{align*}
We set $\tau':=\A(\KK')$ and $\epsilon>0$ small enough so that $(\tau',\tau'+\epsilon]$ does not contain critical values of $\A$. If the negative bundle of $\KK'$ is not orientable, we have
\begin{align*}
H^*_{S^1}(\{\A<\tau'+\epsilon\},\{\A<\tau'\};\Q)=0.
\end{align*}
If, instead, the negative bundle of $\KK'$ is orientable, we have
\begin{align*}
H^*_{S^1}(\{\A<\tau'+\epsilon\},\{\A<\tau'\};\Q)
\cong
H^{*-\ind(\KK')}_{S^1}(\KK';\Q)
\cong
H^{*-\ind(\KK')}(\KK'/S^1;\Q);
\end{align*}
Since $\KK'$ is a closed manifold, its quotient $\KK'/S^1$ is a closed orbifold, and in particular $H^{d}(\KK'/S^1;\Q)$ vanishes for all degrees $d\geq\dim(\KK')=\nul(\KK')$. Therefore, in both cases, we conclude that 
\begin{align*}
H^{d}_{S^1}(\{\A<\tau'+\epsilon\},\{\A<\tau'\};\Q)=0,\qquad \forall d\geq \ind(\KK)-1,
\end{align*}
and thus 
\begin{align}
\label{e:homology_less_tau}
H^{d}_{S^1}(\{\A<\tau\};\Q)=0,\qquad \forall d\geq \ind(\KK)-1.
\end{align}
If we set $i:=(\mu-n)/2=\ind(\KK)/2$, which is a non-negative integer according to Lemma~\ref{l:Besse_indices}, this readily implies that $\fr(\{\A<\tau\})<i$, and therefore
\begin{align}
\label{e:c_i}
c_i(\Sigma)\geq\tau>c_{i-1}(\Sigma).
\end{align}

Analogously, for all connected components $\KK''\subset\crit(\A)$ with $\tau'':=\A(\KK'')>\tau$ we have $\ind(\KK'')>\ind(\KK)+\nul(\KK)=\ind(\KK)+2n-1$ and, if $\epsilon>0$ is small enough so that $(\tau'',\tau''+\epsilon]$ does not contain critical values of $\A$,
\begin{align*}
H^{d}_{S^1}(\{\A<\tau''+\epsilon\},\{\A<\tau''\};\Q)=0,\qquad \forall d\leq \ind(\KK)+2n-1.
\end{align*}
Therefore, for all sufficiently small $\epsilon>0$,
\begin{align*}
H^{d}_{S^1}(\{\A<\infty\},\{\A<\tau+\epsilon\};\Q)=0,\qquad \forall d\leq \ind(\KK)+2n-1,
\end{align*}
which implies that the homomorphism induced by the inclusion
\begin{align*}
H^{\ind(\KK)+2n-2}_{S^1}(\{\A<\infty\};\Q)\to H^{\ind(\KK)+2n-2}_{S^1}(\{\A<\tau+\epsilon\};\Q)
\end{align*}
is injective. This implies that $\fr(\{\A<\tau+\epsilon\})\geq i+n-1$ and thus
\begin{align}\label{e:c_i+n-1}
c_{i+n-1}(\Sigma)\leq\tau. 
\end{align}
Since $c_i(\Sigma)\leq c_{i+n-1}(\Sigma)$, Equations~\eqref{e:c_i} and~\eqref{e:c_i+n-1} imply that 
\[
c_{i-1}(\Sigma)< c_i(\Sigma)=c_{i+n-1}(\Sigma)=\tau.\]
Finally, since $\KK$ is simply connected (being homeomorphic to $\Sigma$), its negative bundle is orientable, and we have
\begin{align*}
H^*_{S^1}(\{\A<\tau+\epsilon\},\{\A<\tau\};\Q)
\cong
H^{*-2i}_{S^1}(\KK;\Q)
\cong
H^{*-2i}(\KK/S^1;\Q);
\end{align*}
since $H^{d}(\KK/S^1;\Q)$ vanishes for all degrees $d\geq\dim(\KK)=\nul(\KK)=2n-1$, we have
\begin{align*}
H^{2(i+n)}_{S^1}(\{\A<\tau+\epsilon\},\{\A<\tau\};\Q)=0.
\end{align*}
This, together with~\eqref{e:homology_less_tau}, implies that $c_{i+n}(\Sigma)>\tau$.
\end{proof}

\subsection{The spectral invariants of ellipsoids}
\label{ss:ellipsoids}
In order to justify the claims made in Example~\ref{ex:ellipsoids}, we compute the spectral invariants of general ellipsoids
\begin{align*}
 E(a)=\left\{ z=(z_1,...,z_n)\in\R^{2n}\ \left|\ \sum_{h=1}^n \frac{|z_h|^2}{a_h}=\frac1\pi \right.\right\},
\end{align*}
where $a=(a_1,...,a_n)$ and $0<a_1\leq...\leq a_n<\infty$. We recall that the associated Reeb flow is given by
\begin{align*}
\phi_R^t:E(a)\to E(a),
\qquad
\phi_R^t(z)=(e^{J2\pi t/a_1}z_1,...,e^{J2\pi t/a_n}z_n ).
\end{align*}
Let $\tau_1<\tau_2<\tau_3<...$ be the elements of the action spectrum $\sigma(E(a))$ enumerated in increasing order. Notice that each $\tau_j$ is a positive multiple of some parameter~$a_h$.

We fix $\alpha\in(1,2)$, and consider the spectral invariants $c_i(E(a))$ defined by means of the Clarke's action functional $\Psi$ and its renormalized version $\A$ associated to the positively $\alpha$-homogeneous Hamiltonian $H:\R^{2n}\to[0,\infty)$ such that $H^{-1}(1)=E(a)$. The following proposition is certainly well-known to the experts, and its analogue for the Ekeland-Hofer capacities was proved in \cite[Section~III]{Ekeland:1990aa}.

\begin{prop}
For all $i\geq1$, the spectral invariant $c_{i-1}(E(a))$ is the $i$-th element in the sequence
\begin{align}
\label{e:spec_ellipsoid}
\underbrace{\tau_1,...,\tau_1}_{\times m_1},\underbrace{\tau_2,...,\tau_2}_{\times m_2},\underbrace{\tau_3,...,\tau_3}_{\times m_3},...
\end{align}
where $m_j$ is the number of parameters $a_h$ having $\tau_j$ as a positive multiple. 
\end{prop}

\begin{proof}
The expression of the Reeb flow $\phi_R^t$ readily implies that, for every $j\geq0$, the stratum of $\tau_j$-periodic orbits 
\[\Sigma_j:=\fix(\phi_R^{\tau_j})\subset E(a)\] 
is  an ellipsoid of dimension
$\dim(\Sigma_j)= 2m_j-1$.
The critical set \[\KK_j:=\crit(\A)\cap\A^{-1}(\tau_j)\] is $S^1$-equivariantly homeomorphic to the ellipsoid $\Sigma_j$, where the $S^1$ action on $\Sigma_j$ is given by the renormalized Reeb flow $t\mapsto\phi_R^{\tau_j t}$. In particular, 
\begin{align*}
H^d_{S^1}(\KK_j;\Q)  \cong H^d_{S^1}(\Sigma_k;\Q) \cong
\left\{
  \begin{array}{@{}ll}
    \Q, & \mbox{if $d\in\{0,2,...,2m_j-2\}$,}\\ 
    0, & \mbox{otherwise}. 
  \end{array}
\right.
\end{align*}
The Morse indices of each $\KK_j$ are given by
\begin{align*}
\ind(\KK_j) & = 2\sum_{h=1}^n \left( \left\lceil\frac{\tau_j}{a_h}\right\rceil -1\right),\\
\nul(\KK_j) & = \dim(\KK_j) = 2m_j-1.
\end{align*}
These values can be obtained by means of the relations \eqref{e:Morse_CZ} and the well-known computation of the Conley-Zehnder indices of the periodic Reeb orbits of $\phi_R^t$ (see, e.g., \cite[Section~I.7]{Ekeland:1990lc}. 
Notice in particular that
\begin{align*}
\ind(\KK_{j+1}) 
&=
\ind(\KK_j) + 2\sum_{h=1}^n \left( 
\left\lceil\frac{\tau_{j
+1}}{a_h}\right\rceil - \left\lceil\frac{\tau_{j
}}{a_h}\right\rceil
\right)\\
&= 
\ind(\KK_j) + 2 m_j = \ind(\KK_j)+\nul(\KK_j)+1, 
\end{align*}
and therefore 
\begin{align*}
\ind(\KK_j)=2\sum_{0<h<j}m_h.
\end{align*}
This readily implies that the functional $\A$ is perfect for the $S^1$-equivariant cohomology, that is, for all $\tau>0$ we have
\begin{equation}
\label{e:sublevel_cohomology}
\begin{split}
H_{S^1}^{*}(\{\A<\tau\};\Q)
& \cong
\bigoplus_{\tau_j< \tau} H_{S^1}^{*-\ind(\KK_j)}(\KK_j;\Q)\\
& \cong
\bigoplus_{\tau_j< \tau} H_{S^1}^{*-2(m_1+...+m_{j-1})}(\KK_j;\Q). 
\end{split}
\end{equation}
Notice that the summands in the right-hand side are supported in complementary even degrees. This  readily implies that $c_i(E(a))=\tau_j$ where $j$ is such that 
\[
2(m_1+...+m_{j-1})\leq 2i\leq 2(m_1+...+m_{j}).
\qedhere
\]
\end{proof}

\section{Restricted contact type hypersurfaces of symplectic vector spaces}

\label{s:EH_capacities}

\subsection{Equivariant spectral invariants}

Let us recall the construction of equivariant spectral values of asymptotically quadratically-convex Hamiltonians, which is due to Ekeland and Hofer, \cite{Ekeland:1990aa}. We denote by $\HH_m$ the space of smooth Hamiltonians $H:\R^{2n}\to[0,\infty)$ such that $\R^{2n}\setminus\supp(H)\neq\varnothing$ and $H(z)=(m+\tfrac12)\pi\|z\|^2$ for some integer $m\geq1$ outside a compact set. 

Recall that the Hilbert space $H^{1/2}(S^1,\R^{2n})=W^{1/2,2}(S^1,\R^{2n})$, where $S^1=\R/\Z$, splits as a direct sum 
\[H^{1/2}(S^1,\R^{2n})=\bigoplus_{k\in\Z} E_k,\] where each $E_k$ is the $2n$-dimensional vector space containing those $\gamma\in C^\infty(S^1,\R^{2n})$ of the form $\gamma(t)=\exp^{J2\pi k t}\gamma(0)$. We consider the orthogonal projections
\begin{align*}
P^\pm:H^{1/2}(S^1,\R^{2n})\to E^\pm:=\bigoplus_{\pm k>0}E_k.
\end{align*}
For each $\gamma\in H^{1/2}(S^1,\R^{2n})$, we will write 
$\gamma^\pm=P^\pm\gamma$ and $\gamma^0=\gamma-\gamma^+-\gamma^-$. 
The 1-periodic orbits of the Hamiltonian flow of any $H\in\HH_m$ are precisely the critical points of the action functional $\Phi_H: H^{1/2}(S^1,\R^{2n})\to\R$ given by
\begin{align*}
\Phi_H(\gamma)
=
\underbrace{\tfrac12\Big( \|\gamma^+\|_{H^{1/2}}^2 - \|\gamma^-\|_{H^{1/2}}^2 \Big)}_{=:A(\gamma)}
-
\underbrace{\int_{S^1} H(\gamma(t))\,\diff t}_{=:B_H(\gamma)}.
\end{align*}
The quadratic form $A:H^{1/2}(S^1,\R^{2n})\to\R$ is sometimes called the symplectic action, since on the subspace of smooth $\gamma\in C^\infty(S^1,\R^{2n})$ it is given by
\begin{align*}
A(\gamma) = \int_\gamma \Lambda,
\end{align*}
where $\Lambda$ is any primitive of the standard symplectic form $\omega$ of $\R^{2n}$.
The functional $\Phi_H$ is smooth (see \cite[Appendix~4]{Hofer:1994bq}) and satisfies the Palais--Smale condition.
Since $H$ is autonomous, $\Phi_H$ is invariant under the $S^1$ action on $H^{1/2}(S^1,\R^{2n})$ given by
\begin{align*}
t\cdot\gamma=\gamma(t+\cdot),\qquad\forall t\in S^1,\ \gamma\in H^{1/2}(S^1,\R^{2n}).
\end{align*}

Ekeland--Hofer's spectral values for $H$ are min-max values of $\Phi_H$ over a suitable family of $S^1$-invariant subspaces. In order to define such a family, they needed to introduce a group $\Gamma$ of $S^1$-equivariant homeomorphisms of $H^{1/2}(S^1,\R^{2n})$ which is large enough to contain the time-$t$ maps of the anti-gradient flow of $\Phi_H$; the group $\Gamma$ consists of those homeomorphisms $\psi:H^{1/2}(S^1,\R^{2n})\to H^{1/2}(S^1,\R^{2n})$ of the form
\begin{align*}
\psi(\gamma) = e^{F^+(\gamma)}\gamma^+ + \gamma^0 + e^{F^-(\gamma)}\gamma^- + K(\gamma),
\end{align*}
where $F^\pm:H^{1/2}(S^1,\R^{2n})\to\R$ are $S^1$-invariant continuous functions mapping bounded sets to bounded sets, and $K:H^{1/2}(S^1,\R^{2n})\to H^{1/2}(S^1,\R^{2n})$ is an $S^1$-equivariant, possibly non-linear, compact, continuous map (here, by a compact map we mean a map such that the image of any bounded set is pre-compact); moreover, $F^+(\gamma)=F^-(\gamma)=0$ and $K(\gamma)=0$ if $A(\gamma)<0$ or if $\|\gamma\|_{H^{1/2}}$ is larger than some positive constant depending on $\psi$.

Let $X\subset H^{1/2}(S^1,\R^{2n})$ be an $S^1$-invariant subset. The \textbf{Ekeland--Hofer index}  is defined by
\begin{align*}
\eh(X):=\inf_{\psi\in\Gamma} \fr(\psi(X)\cap S^+),
\end{align*}
where $S^+$ is the unit sphere of the Hilbert subspace $E^+\subset H^{1/2}(S^1,\R^{2n})$, and $\fr(\cdot)$ denotes the Fadell--Rabinowitz index (see Section~\ref{s:Fadell_Rabinowitz}). Indices of this kind, in an abstract setting, were first investigated by Benci \cite{Benci:1982aa}.
The $i$-th \textbf{Ekeland--Hofer spectral invariant} is the min-max of $\Phi_H$ over the family of $S^1$-invariant subsets $X\subseteq H^{1/2}(S^1,\R^{2n})$  with $\eh(X)\geq i$, i.e.
\begin{align*}
c_i(H) := \inf \left.\left\{ \sup_X \Phi_H\ \right|\ \eh(X)\geq i \right\}.
\end{align*}
Due to the special form of the Hamiltonians in $\HH_m$, it turns out that
\begin{align*}
0<c_0(H)\leq c_1(H)\leq ...\leq c_{n(m-1)-1}(H) \leq c_{n(m-1)}(H)<\infty,
\qquad\forall H\in\HH_m,
\end{align*}
and every such spectral invariant is a critical value of $\Phi_H$.

The following statement is a version of the Lusternik-Schnirelmann theorem for the Ekeland--Hofer spectral invariants.

\begin{lem}
\label{l:LS_EH}
For each $H\in\HH_m$ and integers $0\leq i_1\leq i_2\leq n(m-1)$ such that $c_{i_1}(H)=c_{i_2}(H)=:c$, every $S^1$-invariant  neighborhood $\UU\subseteq H^{1/2}(S^1,\R^{2n})$ of the critical set $\crit(\Phi_H)\cap\Phi_H^{-1}(c)$ has Fadell--Rabinowitz index $\fr(\UU)\geq i_2-i_1$.
\end{lem}

\begin{proof}
Let $X$ be a topological space equipped with an $S^1$ action, and $A,B,C\subset X$ three $S^1$-invariant open subsets such that $\overline B\subset C$. The subadditivity property~\eqref{e:subadditivity_FR} of the Fadell--Rabinowitz index implies
\begin{align*}
\fr(A)
\leq
\fr((A\setminus\overline B)\cup C)
\leq
\fr(A\setminus\overline B)+\fr(C)+1.
\end{align*}
Following Benci \cite{Benci:1982aa}, this can be applied to the Ekeland--Hofer index as follows: if $\mathcal{U}_1,\mathcal{U}_2,\mathcal{Y}\subset H^{1/2}(S^1,\R^{2n})$ are three $S^1$-invariant open subsets such that $\overline{\mathcal{U}_1}\subset \mathcal{U}_2$, $\fr(\mathcal{U}_2)<\infty$, and $\eh(\mathcal{Y})<\infty$, then
\begin{equation}
\label{e:subadditivity_EH}
\begin{split}
\eh(\mathcal{Y}) 
&\geq
\eh(\mathcal{Y}\setminus \overline{\mathcal{U}_1}) \\
&=
\inf_{\psi\in\Gamma}
\fr\big((\psi(\mathcal{Y}\cup \mathcal{U}_1)\cap S^+)\setminus\psi(\overline{\mathcal{U}_1})\big)\\
&\geq
\inf_{\psi\in\Gamma}
\Big(
\fr\big(\psi(\mathcal{Y}\cup \mathcal{U}_1)\cap S^+\big) - \fr(\psi(\mathcal{U}_2)) -1
\Big)\\
&=
\inf_{\psi\in\Gamma}
\Big(
\fr\big(\psi(\mathcal{Y}\cup \mathcal{U}_1)\cap S^+) - \fr(\mathcal{U}_2) -1
\Big)\\
&=
\eh(\mathcal{Y}\cup \mathcal{U}_1)-\fr(\mathcal{U}_2)-1.
\end{split}
\end{equation}

Let us now consider the setting of the statement, with $c_{i_1}(H)=c_{i_2}(H)=:c$. We can assume that some $S^1$-invariant open neighborhood $\mathcal{V}\subset H^{1/2}(S^1,\R^{2n})$ of the critical set $\KK:=\crit(\Phi_H)\cap\Phi_H^{-1}(c)$ has finite Fadell--Rabinowitz index $\fr(\mathcal{V})<\infty$ (for otherwise the lemma already follows).
We fix an arbitrary $S^1$-invariant open neighborhood $\mathcal{U}_2\subseteq \mathcal{V}$ of $\KK$, and a smaller $S^1$-invariant open neighborhood $\mathcal{U}_1\subset\overline{\mathcal{U}_1}\subset \mathcal{U}_2$ of $\KK$. Since $\Phi_H$ satisfies the Palais--Smale condition, there exists $\epsilon>0$ small enough so that the anti-gradient flow of $\Phi_H$ deforms the sublevel set $\{\Phi_H<c+\epsilon\}$ into $\{\Phi_H<c\}\cup \mathcal{U}_1$. This, together with the equality $c=c_{i_2}(H)$, implies 
\begin{align*}
\eh(\{\Phi_H<c\}\cup \mathcal{U}_1)
\geq
\eh(\{\Phi_H<c+\epsilon\})
\geq
i_2.
\end{align*}
Moreover, since $c=c_{i_1}(H)$, we have
\begin{align*}
\eh(\{\Phi_H<c\})
\leq
i_1-1.
\end{align*}
We can now apply Equation~\eqref{e:subadditivity_EH} with $\mathcal{Y}=\{\Phi_H<c\}$, and obtain
\begin{align*}
&\fr(\mathcal{U}_2)
 \geq
 \eh(\{\Phi_H<c\}\cup \mathcal{U}_1) - \eh(\{\Phi_H<c\}) -1 \geq i_2-i_1.
 \qedhere
\end{align*}
\end{proof}

\subsection{The Ekeland--Hofer capacities}

For any bounded subset $B\subset \R^{2n}$ and integer $m\geq 1$, we consider the family of Hamiltonians
\begin{align*}
 \HH_m(B):=\big\{ H\in\HH_m\ \big|\ \supp(H)\cap \overline B=\varnothing \big\},
\end{align*}
and we set
\begin{align*}
\HH(B):=\bigcup_{m\geq1} \HH_m(B).
\end{align*}
The $i$-th \textbf{Ekeland--Hofer capacity}\footnote{Our index $i$ for the Ekeland--Hofer capacity is shifted by one with respect to the index in the original reference \cite{Ekeland:1990aa}: namely, our  $c_0(\cdot)$ corresponds to $c_1(\cdot)$ in \cite{Ekeland:1990aa}.} of the bounded set $B$ is defined as
\begin{align*}
c_i(B) = \inf_{H\in\HH(B)} c_i(H).
\end{align*}
This is indeed a symplectic capacity, i.e., it satisfies the monotonicity property ($c_i(A)\leq c_i(B)$ if there exists a symplectic embedding $\psi:A\hookrightarrow B$), the conformality property ($c_i(r B)=r^2 c_i(B)$ for all $r\in\R\setminus\{0\}$), and is non-trivial ($c_i(B^{2n}(1))=c_i(B^2(1)\times B^{2n-2}(R))=(i+1)\pi$ for all $R\geq1$). 

We consider a compact subset $B\subset\R^{2n}$ whose boundary $\Sigma=\partial B$ is smooth and of restricted contact type. A remarkable feature of the Ekeland--Hofer capacities is that they are action selectors, i.e.
\begin{align}\label{e:EH_representation}
c_i(B)=c_i(\Sigma)\in \sigma(\Sigma);
\end{align}
see \cite[Prop.~2]{Ekeland:1990aa}.

\subsection{A sufficient condition for the Besse property}
Let $\Sigma\subset\R^{2n}$ be a compact hypersurface of restricted contact type, and $B$ the closure of the bounded connected component of $\R^{2n}\setminus\Sigma$. In the proof of Theorem~\ref{t:EH_Besse}, we utilize the 
sequence of Hamiltonians $H_m\in\HH_m(B)$ from \cite{Ekeland:1989aa} used in the proof of~\eqref{e:EH_representation}. Since we need to make minor modifications in the construction, we include its full details for the reader's convenience.

We fix a primitive $\Lambda$ of the standard symplectic form $\omega$ of $\R^{2n}$ which restricts to a contact form $\lambda=\Lambda|_\Sigma$, and we denote by $R$ the Reeb vector field of $(\Sigma,\lambda)$, and by $\phi_R^t$ the associated Reeb flow. We now consider the associated Liouville vector field $V$ on $\R^{2n}$, which is uniquely defined by $\Lambda=\omega(V,\cdot)$, and denote by $\phi_V^s$ its flow. Since $\Lie_V\Lambda=\Lambda$, we have $(\phi_V^s)^*\Lambda=e^s\Lambda$. In order to simplify the notation, let us assume without loss of generality that $\phi_V^s|_\Sigma$ is defined for all $s\in[0,2]$. The hypersurface $\Sigma_s:=\phi_V^s(\Sigma)$ is again of restricted contact type, and its Reeb vector field is \[R_s(\phi_V^s(z))=e^{-s}\diff\phi_V^s(z)R(z).\] In particular, 
\[\sigma(\Sigma_s)=e^s\sigma(\Sigma),
\qquad\forall s\in[0,2].\]

Let $r>0$ be large enough so that 
\begin{align*}
\Sigma_2=\phi_V^2(\Sigma) \subset B^{2n}(r).
\end{align*}
We fix $b>(m+\tfrac12)\pi r^2$ and $k\geq\max\{4,b^{-1}\}$. Since the action spectrum $\sigma(\Sigma)$ is nowhere dense, we can find
\begin{align*}
\tau\in\big[(b-\tfrac1k)(e^{3/k}-e^{2/k})^{-1},(b-\tfrac1{2k})(e^{3/k}-e^{2/k})^{-1}\big]\setminus\sigma(\Sigma).
\end{align*}
We consider a smooth monotone increasing function $\phi=\phi_{b,k}:[0,\infty)\to[0,b]$ such that 
\begin{align*}
&\phi|_{[0,1/k]}\equiv0,\\
&\tfrac{\diff^3}{\diff s^3}\phi(s)>0,\quad \forall s\in\big(\tfrac 1k,\tfrac2k\big],\\
&\phi(\tfrac2k)<\tfrac1{2k},\\
&\phi|_{[4/k,\infty)}\equiv b,\\
&\phi(s)= \tau e^{s} - \tau e^{2/k} + \phi(\tfrac2k),
\quad
\forall s\in\big[\tfrac2k,\tfrac3k\big].
\end{align*}
The condition on the third derivative of $\phi$ guarantees that
\begin{align}
\label{e:monotonicity_phi}
 \tfrac{\diff}{\diff s}\big( \dot\phi(s)-\phi(s) \big)>0,
 \qquad
 \forall s\in \big(\tfrac1k,\tfrac2k\big],
\end{align}
while the bounds on $\tau$ imply that
\begin{align*}
b - \tfrac1k < \phi(s) \leq b,
\qquad
\forall s\in[\tfrac3k,\infty).
\end{align*}
Next, we consider a smooth convex function $g=g_{b,m} \colon [0,\infty)\to[b,\infty)$ such that $g|_{[0,r]}\equiv b$, $g(s)\geq(m+\tfrac12)\pi s^2$ for all $s>r$, and $g(s)=(m+\tfrac12)\pi s^2$ for all $s$ large enough. We set 
\[B_s:= \phi_V^s(B) \subset \R^{2n},\] 
which is the compact subset with boundary $\partial B_s=\Sigma_s$.
We define the Hamiltonian
\begin{align}
\label{e:Hamiltonian}
H=H_{b,k,m}\in\HH_m(B),
\qquad
H(z)
=
\left\{
  \begin{array}{@{}lll}
    0, &  & \mbox{if }z\in B, \\ 
    \phi(s), &  & \mbox{if }z\in \Sigma_s,\ s\in[0,1] \\ 
    g(|z|), &  & \mbox{if }z\not\in B_1. \\ 
  \end{array}
\right.
\end{align}

\begin{lem}
\label{l:periodic_orbits_H_R}
For each positive critical value $c>0$ of the action functional $\Phi_H$ there exists
\begin{align*}
s=s(c)\in (\tfrac1k,\tfrac2k\big]\cup [\tfrac3k,\tfrac4k\big)
\end{align*}
such that every critical point $\gamma\in\crit(\Phi_H)\cap\Phi_H^{-1}(c)$ lies on the energy hypersurface $\Sigma_{s}$. In particular, 
\begin{align}
\label{e:c_phi}
c=\dot\phi(s)-\phi(s),
\end{align}
and the curve
$\zeta(t) 
:=
\phi_V^{-s}(\gamma(te^{s}/\dot\phi(s)))$
is a closed Reeb orbit of $(\Sigma,\lambda)$ with period $A(\zeta)=\dot\phi(s)e^{-s}$.
\end{lem}

\begin{proof}
Let us fix a critical point $\gamma\in\crit(\Phi_H)\cap\Phi_H^{-1}(0,\infty)$.
Since the Hamiltonian $H$ is autonomous, $\gamma$ is contained in a level set of $H$. Clearly, $\sigma$ cannot intersect $B_{1/k}$ nor $B^{2n}(r)\setminus B_{4/k}$, for otherwise it would be a constant curve with $H(\gamma)\geq0$ and associated critical value $\Phi_H(\gamma)\leq0$. Analogously, $\gamma$ cannot intersect the complement of $B^{2n}(r)$, for otherwise $H(\gamma)=g(|\gamma|)$, and we would still have
\begin{align*}
\Phi_H(\gamma)
&=
\tfrac12 \dot g(|\gamma|)|\gamma|-g(|\gamma|)
\leq 
\tfrac12(m+\tfrac12)\pi2|\gamma|\,|\gamma|-(m+\tfrac12)\pi|\gamma|^2=0.
\end{align*}
Therefore, $\gamma$ must be contained in $B_{4/k}\setminus B_{1/k}$, which is foliated by the restricted contact type energy hypersurfaces $\Sigma_s$, $s\in[\tfrac1k,\tfrac4k]$. Notice that on $\Sigma_s$ the Hamiltonian vector field $J\nabla H$ is given by
\begin{align*}
J\nabla H|_{\Sigma_s} = \dot\phi(s)R_s
\end{align*}
In particular, if $\gamma$ lies on $\Sigma_{s_\gamma}$, the curve 
\[\zeta(t)=\phi_V^{-s_\gamma}(\gamma(t e^{s_\gamma}/\dot\phi(s_\gamma))\] 
is a closed Reeb orbit of $(\Sigma,\lambda)$ with period $A(\zeta):=\dot\phi(s_\gamma)e^{-s_\gamma}$. This readily implies that $s_\gamma\not\in[\tfrac2k,\tfrac3k]$, for otherwise we would have $A(\zeta)=\tau e^{s_\gamma}e^{-s_\gamma}=\tau$, contradicting the fact that $\tau\not\in\sigma(\Sigma)$. Finally, notice that the critical value of $\gamma$ is
\begin{align*}
 \Phi_H(\gamma) = \dot\phi(s_\gamma) - \phi(s_\gamma).
\end{align*}
By~\eqref{e:monotonicity_phi}, the function $s\mapsto\dot\phi(s)-\phi(s)$ is strictly monotone increasing on $\big(\tfrac1k,\tfrac2k\big]$, and therefore $s_\gamma$ only depends on the critical value $\Phi_H(\gamma)$.
\end{proof}

From now on, we will assume that the action spectrum $\sigma(\Sigma)$ is discrete.

\begin{lem}
\label{l:filtering_orbits}
For each integer $m\geq4$ there exist 
\[b>(m+\tfrac12)\pi r^2,
\qquad
k\geq\max\{m,b_{m}^{-1}\}\] 
such that the Hamiltonian 
\begin{align}
\label{e:H_m}
 H_m:=H_{b,k,m} \in\HH_m(B)
\end{align}
has the following property: there exist
\begin{align*}
 s_i(H_m)\in (\tfrac1{k},\tfrac2{k}\big],
 \qquad
 \forall i\in\{0, \dotsc,n(m-1)\}
\end{align*}
such that every critical point $\gamma\in\crit(\Phi_{H_m})\cap\Phi_{H_m}^{-1}(c_i(H_m))$ lies on the energy hypersurface $\Sigma_{s_i(H_m)}$.
\end{lem}

\begin{proof}
We fix an integer $m\geq4$, and we remove it from the notation as much as possible.
Let $\phi_{b,k}$ be the function entering the definition of $H_{b,k,m}$ as in~\eqref{e:Hamiltonian}.
By Lemma~\ref{l:periodic_orbits_H_R}, there exists 
\[
s=s_{i,b,k}\in(\tfrac1k,\tfrac2k\big]\cup [\tfrac3k,\tfrac4k\big)\] 
such that every critical point
\begin{align*}
\gamma=\gamma_{i,b,k}\in\crit(\Phi_{H_{b,k,m}})\cap\Phi_{H_{b,k,m}}^{-1}(c_i(H_{b,k,m}))
\end{align*}
lies on the energy hypersurface $\Sigma_{s}$, and the curve 
\[\zeta=\zeta_{i,b,k}\] given by
$\zeta(t) = 
\phi_V^{-s}(\gamma(te^s/\dot\phi_{b,k}(s)))$
is a closed Reeb orbit of $(\Sigma,\lambda)$ with period $A(\zeta)=\dot\phi_{b,k}(s)e^{-s}$. Equation~\eqref{e:c_phi}, together with $A(\gamma)=e^s A(\zeta)$, implies that
\begin{align}
\label{e:symplectic_action_bounds}
A(\zeta) - e^{-s} c_i(H_{b,k,m}) = e^{-s} \phi_{b,k}(s) 
\in
\big[0,e^{-s}\tfrac1{2k}\big] \cup \big[ e^{-s}(b-\tfrac1{k}),e^{-s}b\big].
\end{align}

Notice that, if we fix an integer $k_1\geq\max\{4,b^{-1}\}$, for all integers $k_2>k_1$ sufficiently large we have $H_{b,k_1,m}<H_{b,k_2,m}$ pointwise. This implies that the limit 
\begin{align*}
c_i(b):= \lim_{k\to\infty} c_i(H_{b,k,m})
\end{align*}
exists and is finite and bounded from below by the Ekeland--Hofer capacity $c_i(B)$. Moreover, for each $s\in(0,1]$, we have that $H_{b,k,m}|_{\Sigma_s}\to b$ as $k\to\infty$, and therefore $c_i(b)$ can also be characterized by
\begin{align*}
c_i(b) = \inf\left\{ c_i(H)\  \left|\ H\in\HH_m(B),\ \max_{B^{2n}(r)}H\leq b \right.\right\}.
\end{align*}
In particular, $b\mapsto c_i(b)$ is continuous and non increasing. Set
\begin{align*}
 \overline c_i:=\lim_{b\to\infty} c_i(b).
\end{align*}
The uniform bound on $c_i(b)$, together with~\eqref{e:symplectic_action_bounds}, implies that the limits
\begin{align*}
a_{i,b}':=\liminf_{k\to\infty} A(\zeta_{i,b,k})\leq
\limsup_{k\to\infty} A(\zeta_{i,b,k})=:a_{i,b}''
\end{align*}
are both finite, and 
\begin{align}
\label{e:lim_inf_sup}
 a_{i,b}',a_{i,b}''\in\{c_i(b),c_i(b)+b\}\cap\sigma(\Sigma).
\end{align}

We claim that $c_i(b)\in\sigma(\Sigma)$ for all $b$ large enough. Indeed, arguing by contradiction assume that $c_i(b)\not\in\sigma(\Sigma)$ for arbitrarily large $b$. Since $\sigma(\Sigma)$ is assumed to be discrete and since $c_i(b)$ converges monotonically to $c_i$ as $b\to\infty$, we indeed have that $c_i(b)\not\in\sigma(\Sigma)$ for all $b$ large enough, say for $b\geq\overline b$. This, together with~\eqref{e:lim_inf_sup}, implies that $a_{i,b}''=c_i(b)+b$ for all $b\geq\overline b$. In particular $c_i(b)+b\in\sigma(\Sigma)$ for all $b\geq\overline b$. Since the action spectrum $\sigma(\Sigma)$  is discrete, the continuous function $b\mapsto c_i(b)+b$ should be constant for $b\geq\overline b$. However, we already know that $c_i(b)+b\to\infty$ as $b\to\infty$. This contradiction proves the desired claim. 

Since $b\mapsto c_i(b)$ is continuous and takes values inside the discrete action spectrum $\sigma(\Sigma)$ for all $b$ large enough, say for $b\geq\overline b$, we have that 
\[c_i(b)=\overline c_i,\qquad\forall b\geq\overline b.\] 
Once again, thanks to the discreteness of the action spectrum $\sigma(\Sigma)$, we can fix an arbitrarily large $b>(m+\tfrac12)$ such that
\begin{align*}
\overline c_i+b\not\in\sigma(\Sigma),\qquad\forall i=0, \dotsc,n(m-1),
\end{align*}
which implies 
$a_{i,b}'=a_{i,b}''=c_i(b)$. Therefore, for this value of $b$ and for all integers $k$ large enough, Equation~\eqref{e:symplectic_action_bounds} implies that
\[e^{-s_{i,b,k}} \phi_{b,k}(s_{i,b,k}) 
\in
\big[0,e^{-s_{i,b,k}}\tfrac1{2k}\big],\] 
and thus that $s_{i,b,k}\in(\tfrac1k,\tfrac2k\big]$.
\end{proof}

\begin{proof}[Proof of Theorem~\ref{t:EH_Besse}]
We consider the sequence of Hamiltonians $H_m\in\HH_m(B)$ provided in~\eqref{e:H_m}, which satisfies
\begin{align*}
c_i(B)=c_i(\Sigma)=\lim_{m\to\infty} c_i(H_m).
\end{align*}
We denote by $\phi_m$ the function that enters the definition of $H_m$ as in~\eqref{e:Hamiltonian}, so that $H_m|_{\Sigma_s}=\phi_m(s)$ for all $s\in[0,1]$. We set $\tilde m:=1+(i+n-1)n^{-1}$, so that for every integer $m\geq\tilde m$ we have $i+n-1\leq n(m-1)$ and $c_{i+n-1}(H_m)<\infty$.

By Lemmas~\ref{l:periodic_orbits_H_R} and~\ref{l:filtering_orbits}, for each $m\geq \tilde m$ and $j\in\{0, \dotsc,i+n-1\}$ there exists 
\[s_{m,j}:=s_j(H_m)\in(0,\tfrac2m)\] 
such that every critical point $\gamma\in\crit(\Phi_H)\cap\Phi_H^{-1}(c_j(H_m))$ lies on the energy hypersurface $\Sigma_{s_{m,j}}$ and 
\[
\tau_{m,j}
:=
e^{-s_{m,j}} \dot\phi_m(s_{m,j})
=
e^{-s_{m,j}}(c_j(H_m)+\phi_m(s_{m,j}))
\in\sigma(\Sigma).
\]
Since $0<\phi_m(s_{m,j})\leq\phi_m(\tfrac2m)<\tfrac1{2m}$, we have the limit
\begin{align*}
 \lim_{m\to\infty} \tau_{m,j}=c_j(\Sigma).
\end{align*}
Due to our assumption that the action spectrum $\sigma(\Sigma)$ is discrete, this limit implies that the sequence $\tau_{m,j}$ must stabilize for large $m$ and therefore
\begin{align*}
 \tau_{m,j}=c_j(\Sigma),\qquad\forall m\geq m_j.
\end{align*}
By Equation~\eqref{e:monotonicity_phi}, $\tau_{m,j}$ uniquely determines $s_{m,j}$. Therefore, if we fix an integer $m\geq\max\{m_i,m_{i+n-1}\}$, since 
\[
\tau:=\tau_{m,i}=c_i(\Sigma)=c_{i+n-1}(\Sigma)=\tau_{m,i+n-1},
\] 
we have $s:=s_{m,i}=s_{m,i+n-1}$ and
\begin{align*}
c_i(H_m) 
=
e^{-s} (\dot\phi_m(s)-\phi_m(s))
=
c_{i+n-1}(H_m).
\end{align*}

We now proceed as in the proof of Lemma~\ref{l:neighborhoods}, by choosing a Riemannian metric $g$ on $\Sigma$ such that the Reeb orbits on $\Sigma$ are unit speed geodesics. Consider the compact subset $K:=\fix(\phi_R^\tau)\subset\Sigma$. For each $z\in K$, we denote the corresponding $\tau$-periodic Reeb orbit by $\zeta_z(t)=\phi_R^t(z)$ and set
\begin{align*}
\KK:=\{\zeta_z\ |\ z\in K \}\subset C^{\infty}(\R/\tau\Z,\Sigma).
\end{align*}
As in the proof of Lemma~\ref{l:neighborhoods}, for each open neighborhood $W\subseteq\Sigma$ of $K$, we can find an open neighborhood $\NN\subset H^{1/2}(\R/\tau\Z,\R^{2n})$ of $\KK$ which is homotopy equivalent to $W$. Picking $W$ small, we can make $\NN$ an arbitrarily small neighborhood of $\KK$.

To finish the proof, assume now to the contrary that $\Sigma$ is not Besse, so that $K\neq\Sigma$. Choosing $W\neq\Sigma$, we then obtain 
\[H^{*\geq 2n-1}(\NN;\Q)\cong H^{*\geq 2n-1}(W;\Q)=0.\]
The open subset 
\begin{align*}
\NN_m:= \big\{ \phi_V^{s}(\zeta(\tau \cdot))\ \big|\ \zeta\in\NN \big\}\subset H^{1/2}(S^1,\R^{2n})
\end{align*}
is an arbitrarily small open neighborhood of  $\KK_m:=\crit(\Phi_{H_m})\cap\Phi_{H_m}^{-1}(c_i(H_m))$ with 
\[H^{*\geq 2n-1}(\NN_m;\Q)\cong H^{*\geq 2n-1}(\NN;\Q)=0.\]
By Lemma~\ref{l:bound_FR}, $\KK_m$ admits an $S^1$-invariant neighborhood $\UU\subset H^{1/2}(S^1,\R^{2n})$ with $\fr(\UU)<n-1$. This, together with Lemma~\ref{l:LS_EH}, contradicts the equality $c_i(H_m)=c_{i+n-1}(H_m)$.
\end{proof}

\section{Geodesic flows}
\label{s:geodesic_flows}

\subsection{Equivariant spectral invariants}
\label{ss:geodesic_flows_equivariant_spectral_invariants}

Let $(M,g)$ be a closed Riemannian manifold of dimension $n\geq 2$. Its closed geodesics, parametrized in order to have constant speed and period $1$, are the critical points with positive critical value of the energy functional
\begin{align*}
E:\Lambda M\to[0,\infty),
\qquad
E(\gamma)=\int_{S^1} \|\dot\gamma(t)\|_g^2\,\diff t,
\end{align*}
where $\Lambda M=W^{1,2}(S^1,M)$ and $S^1=\R/\Z$. Once again, we are in an equivariant setting: the circle $S^1$ acts on $\Lambda M$ by time-shift
\begin{align*}
 t\cdot\gamma=\gamma(t+\cdot),
 \qquad
 \forall t\in S^1,\ \gamma\in\Lambda M,
\end{align*}
and the energy $E$ is invariant under this action. The subspace of constant curves $E^{-1}(0)$, which we identify with $M$ with a common abuse of notation, is the set of fix points of this action. Every other critical point $\gamma\in\crit(E)\cap E^{-1}(0,\infty)$ thus belongs to an embedded critical circle $S^1\cdot\gamma\subset\Lambda M$. Moreover, for every integer $m\geq1$, the $m$-th iterate $\gamma^m\in\Lambda M$, which is defined by $\gamma^m(t)=\gamma(mt)$, is also a critical point of $E$ with critical value $E(\gamma^m)=m^2 E(\gamma)$. Therefore, every oriented closed geodesic of $(M,g)$ gives rise to the countable sequence of critical circles $S^1\cdot\gamma^m$ of $E$.

The energy functional $E$ satisfies all the commonly desired assumptions from critical point theory: it is non-negative, smooth, and satisfies the Palais--Smale condition; see e.g.~\cite{Klingenberg:1978so}. As usual, we denote the energy sublevel sets by $\Lambda M^{<b}:=\{E<b\}$, and by $\iota_b:(\Lambda M^{<b},M)\hookrightarrow(\Lambda M,M)$ the inclusion. Every non-zero cohomology class $\mu\in H^*_{S^1}(\Lambda M,\Lambda M;\Q)$ defines an equivariant spectral invariant  
\begin{align*}
c_g(\mu):=\inf\big\{ \sqrt b>0\ \big|\ \iota_b^*\mu\neq0 \big\}\in(0,\infty)
\end{align*}
which is the square root of a positive critical value of $E$, that is, the period of a unit-speed closed geodesic of $(M,g)$.

\subsection{A sufficient condition for the Besse property}

Once again, one of the ingredients for the proof of Theorem~\ref{t:geodesics_Zoll} is the Fadell--Rabinowitz index (see Section~\ref{s:Fadell_Rabinowitz}), which in the context of closed geodesics was first investigated by Rademacher \cite{Rademacher:1994aa}.
We will need the following finiteness property of the index, which is certainly well known to the experts, and holds in a rather general setting.
\begin{lem}
\label{l:finite_FR}
Let $(X,g)$ be a Hilbert manifold equipped with a continuous $S^1$ action such that, for each $t\in S^1$, the action map $\rho_t:X\to X$, $\rho_t(x)=t\cdot x$ is a smooth isometry. Let $K\subset X$ a compact $S^1$-invariant subset such that, for each $x\in K$, the curve $\gamma_x:S^1\to K$, $\gamma_x(t)=\rho_t(x)$ is a smooth immersion. Then, $K$ has an $S^1$-invariant neighborhood $U\subseteq X$ with finite Fadell--Rabinowitz index $\fr(U)<\infty$.
\end{lem}

\begin{proof}
Since every $\rho_t$ is an isometry, the exponential map $\exp$ of the Hilbert manifold $(X,g)$ is $S^1$-equivariant, i.e.,
\[\exp_{t\cdot x}(t\cdot v)=t\cdot\exp_x(v).\] 
Since the $S^1$ action on $K$ is locally free, for each $x\in K$ there exists $\tau=\tau_x\in(0,1]$ such that $\tau\cdot x=x$ and $t\cdot x\neq x$ for all $t\in(0,\tau)$. The $S^1$-orbit 
\[C_x:=S^1\cdot x=\gamma_x([0,\tau])\subseteq K\] 
is a smooth embedded circle in $X$. Consider the normal bundle $N C_x$ and, for each $t\in\R/\tau\Z$, the open ball $B_t\subset N_{t\cdot x}C_{x}$ of radius $\epsilon>0$ around the origin. Notice that $s\cdot B_t=B_{s+t}$ for all $s\in S^1$. The union 
\[B:=\bigcup_{t\in\R/\tau\Z}B_t\]  
forms the open neighborhood of radius $\epsilon$ of the zero-section of the normal bundle $NC_x$. We require $\epsilon>0$ to be small enough so that the exponential map provides a diffeomorphism onto its image
\begin{align*}
\psi:B\to X,
\qquad
\psi(v)=\exp_{t\cdot x}(v),
\quad\forall v\in B_{t}.
\end{align*}
Namely, $U_x:=\psi(B)$ is an $S^1$-invariant tubular neighborhood of $C_x$, which admits an $S^1$-invariant deformation retraction
\begin{align*}
r_s: U_x\to U_x,
\qquad
r_s(\exp_{t\cdot x}(v))=\exp_{t\cdot x}((1-s)v)
\end{align*}
such that $r_0=\id$ and $r_1$ is a retraction onto $C_x$. This implies that 
\begin{align*}
H^{*}_{S^1}(U_x)\cong H^*_{S^1}(C_x)\cong H^*(C_x/S^1)\cong H^*(\mathrm{pt}),
\end{align*}
and in particular $\fr(U_x)=0$. Since $K$ is compact, there exists a finite collection $x_1,\dotsc, x_r\in K$ such that $U:=U_{x_1}\cup \dotsc \cup U_{x_r}$ is an open neighborhood of $K$. By the subadditivity of the Fadell Rabinowitz index (Equation~\eqref{e:subadditivity_FR}), we conclude
\[
\fr(U)\leq\fr(U_1)+1+\fr(U_2)+1+ \dotsc +\fr(U_r)+1=r.
\qedhere
\]
\end{proof}

The following lemma provides a sufficient condition for a Riemannian metric to be Besse and is an equivariant analogue of \cite[Lemma~5.2]{Mazzucchelli:2018pb}, although its proof is somewhat different. We will employ the notation of Section~\ref{s:Fadell_Rabinowitz} and denote by $e_{\Lambda M}\in H^2_{S^1}(\Lambda M;\Q)$ the Euler class of the circle bundle $\Lambda M\times ES^1\to\Lambda M\times_{S^1} ES^1$.

\begin{lem}
\label{l:Besse}
Let $\mu\in H^*_{S^1}(\Lambda M,M;\Q)$ be a cohomology class such that 
$\mu\smile e_{\Lambda M}^{n-1}\neq0$ in $H^{*}_{S^1}(\Lambda M,M;\Q)$. 
If $\ell:=c_g(\mu)=c_g(\mu\smile e_{\Lambda M}^{n-1})$, then $g$ is Besse and $\ell$ is a common period for its unit-speed closed geodesics.
\end{lem}

\begin{proof}
The proof is analogous to the one of Lemma~\ref{l:neighborhoods}, but since the current setting is different we provide the argument in full details. 
We denote by $\phi_t:SM\to SM$ the geodesic flow of $(M,g)$, and set 
\[
\KK:=\crit(E)\cap E^{-1}(\ell^2),
\qquad 
K:= \big\{ (\gamma(0),\dot\gamma(0))\in SM\ \big|\ \gamma\in\KK \big\}.
\] 
Notice that $K$ is invariant under the geodesic flow and 
$\phi_{\ell}|_{K}=\id$. 
The free loop space $\Lambda M$ and its $S^1$ action satisfy the assumptions of Lemma~\ref{l:finite_FR}, which implies that $\KK$ admits an $S^1$-invariant open neighborhood $\XX\subset \Lambda M$ with finite Fadell--Rabinowitz index $\fr(\XX)<\infty$.

Let us assume that $(M,g)$ has at least one geodesic that is not closed, or is closed but not with period $\ell$. Therefore, $SM\setminus K\neq\varnothing$. 
We fix a constant $\epsilon\in(0,\injrad(M,g))$, and an open neighborhood $W\subsetneq SM$ of $K$ which is small enough so that $\dist(\exp_x((\ell-\epsilon)v,x)<\injrad(M,g)$. For each $z=(x,v)\in W$, we set 
$x':=\exp_{x}((\ell-\epsilon)v)$, $v':=\epsilon^{-1}\exp_{x'}^{-1}(x)$, and define the smooth embedding
\begin{align*}
\iota:W\hookrightarrow\Lambda M,
\qquad
\iota(z)(t)
:=
\gamma_{z}(t)
=
\left\{
  \begin{array}{@{}ll}
    \exp_x(tv), & \mbox{if }t\in[0,\ell-\epsilon], \\ 
    \exp_{x'}((t-\ell)v') & \mbox{if }t\in[\ell-\epsilon,\ell]. 
  \end{array}
\right.
\end{align*}
We consider a tubular neighborhood $\WW\subset \Lambda M$ of $\iota(W)$, so that $\iota:W\hookrightarrow\WW$ is a homotopy equivalence, and in particular 
\begin{align*}
H^{*\geq2n-1}(\WW;\Q) \cong H^{*\geq2n-1}(W;\Q) = 0.
\end{align*}
By shrinking $W$ and $\WW$, this latter tubular neighborhood can be made smaller than any given  neighborhood of $\KK$, and in particular contained in the above $S^1$-invariant neighborhood $\XX$ with finite Fadell--Rabinowitz index. We can now apply the abstract Lemma~\ref{l:neighborhoods}, which provides an $S^1$-invariant neighborhood $\UU$ of $\KK$ with Fadell--Rabinowitz index 
\begin{align}
\label{e:bound_fr_U_geodesics}
\ind(\UU)<n-1. 
\end{align}
If we had $\ell:=c_g(\mu)=c_g(\mu\smile e_{\Lambda M}^{n-1})$, the classical Lusternik-Schnirelmann theorem would imply that every $S^1$-invariant neighborhood $\UU\subset \Lambda M$ of $\KK$ has Fadell--Rabinowitz index $\fr(\UU)\geq n-1$, contradicting~\eqref{e:bound_fr_U_geodesics}.
\end{proof}

\subsection{Equivariant Morse theory of the energy functional}
\label{ss:Morse_theory}

Let $M$ be a simply connected Besse manifold of dimension $n\geq 2$ and let $g$ be a Besse Riemannian metric on $M$. The equivariant Morse theory of the associated energy function $E:\Lambda M\to[0,\infty)$ has been thoroughly investigated by Hingston \cite{Hingston:1984hj} and, specifically in the Besse case, by Radeschi and Wilking \cite{Radeschi:2017dz}. In this subsection, we recap those results from \cite{Radeschi:2017dz} which will be needed later on. 

The Besse condition on the Riemannian metric implies that $E$ is Morse-Bott, that is, the set of critical points $\crit(E)$ is a disjoint union of closed manifolds which are transversally non-degenerate. Moreover, $E$ is perfect for the $S^1$-equivariant rational cohomology relative to the constant loops $H^*_{S^1}(-,M;\Q)$, which means that, for all $0<a<b<c\leq\infty$, the inclusion induces the injective and, respectively, surjective homomorphisms
\begin{align*}
H^*_{S^1}(\Lambda M^{<c},\Lambda M^{<b};\Q)
\hookrightarrow
H^*_{S^1}(\Lambda M^{<c},\Lambda M^{<a};\Q)  
\twoheadrightarrow 
H^*_{S^1}(\Lambda M^{<b},\Lambda M^{<a};\Q).
\end{align*}

If $\KK$ is a connected component of $\crit(E)$, we denote by $\ind(\KK)$ its Morse index, and by
$\nul(\KK):=\dim(\ker(\diff^2E|_\KK))$ its nullity\footnote{In the closed geodesics literature, the nullity is often defined as $\dim(\ker(\diff^2E|_\KK))-1$, so that a non-degenerate closed geodesic has an associated critical circle with nullity 0.}. Both indices are well known to be finite; see e.g. \cite{Klingenberg:1978so}.
We assume that $M$ is spin, which implies that every critical manifold of $E$ has an orientable negative bundle, and therefore is homologically visible. This, together with the perfectness of $E$, implies
\begin{align}
\label{e:direct_sum}
H^*_{S^1}(\Lambda M,M;\Q)
\cong
\bigoplus_{\KK} H^{*-\ind(\KK)}(\KK/S^1;\Q),
\end{align}
where the direct sum runs over the space of all connected components $\KK\subset\crit(E)$. The integer
\begin{align*}
i(M):=\min\big\{\ind(\KK)\ \big|\ \KK\in\crit(E)\cap E^{-1}(0,\infty)\big\}
\end{align*}
turns out to be positive and independent of the choice of a Besse Riemannian metric on $M$; indeed, if $M_0$ is the model of $M$ (see Section~\ref{ss:geodesic_flows}), we have 
\begin{align}
\label{e:i(M)}
i(M)
=
i(M_0)
=
\left\{
  \begin{array}{@{}ll}
    n-1, & \mbox{if }M_0=S^n, \\ 
    1, & \mbox{if }M_0=\CP^{n/2}, \\ 
    3, & \mbox{if }M_0=\HP^{n/4}, \\
    7, & \mbox{if }M_0=\CaP^2.
  \end{array}
\right.
\end{align}

Since $g$ is a Besse Riemannian metric, Wadsley's theorem \cite{Wadsley:1975sp} implies that all unit speed closed geodesics have a minimal common period $\ell>0$. Therefore, the critical manifold $\KK:=\crit(E)\cap E^{-1}(\ell^2)$ is diffeomorphic to the unit tangent bundle $SM$ via the evaluation map $\gamma\mapsto(\gamma(0),\dot\gamma(0))$. We denote by $\KK^m:=\{\gamma^m\ |\ \gamma\in\KK\}$ the critical manifold containing the $m$-th iterates of the loops in $\KK$. Bott's index formula \cite{Bott:1956sp} for the critical manifolds $\KK^m$ becomes particularly simple:
\begin{align}
\label{e:Bott}
\ind(\KK^m) = m\,\ind(\KK) + (m-1)(n-1),
\qquad
\nul(\KK^m) = 2n-1.
\end{align}

\subsection{Spectral characterization of Zoll Riemannian metrics}\label{ss:spectral_characterization_Riemannian}
We now assume that the Riemannian metric $g$ is Zoll and that its unit-speed closed geodesics have minimal period $\ell>0$. The critical manifold of non-iterated closed geodesics  
\[\KK:=\crit(E)\cap E^{-1}(\ell^2)\]
is $S^1$-equivariantly homeomorphic to the unit tangent bundle $SM$ equipped with the $S^1$ action provided by the geodesic flow. Every other critical manifold of $E$ with positive critical value is of the form $\KK^m$, for $m\geq2$. Bott's formulas~\eqref{e:Bott} for the Morse indices reduce to
\[
\ind(\KK^m)=m\,i(M)+(m-1)(n-1),
\qquad
\nul(\KK^m)=2n-1
.
\] 
This, together with~\eqref{e:direct_sum}, implies that $H^{*}_{S^1}(\Lambda M,M;\Q)$ has the form~\eqref{e:rel_equiv_hom_loop_space}, and in particular
\begin{align}
\label{e:rank_one}
H^{i(M)}_{S^1}(\Lambda M,M;\Q)\cong H^0(\KK/S^1)\cong \Q.
\end{align}

Let $\NN_m\subset N^-\KK^m$ be the open neighborhood of radius $r>0$ of the zero-section in the negative normal bundle of $\KK^m$. Namely, for every $\gamma\in\KK$, the fiber $(\NN_m)_\gamma$ is the open ball of radius $r>0$ in the negative eigenspace of the Hessian $\diff^2E(\gamma)$. We assume that $r>0$ is small enough so that the exponential map of $\Lambda M$ (with respect to its canonical $S^1$-invariant Riemannian metric induced by $g$) is a diffeomorphism onto its image, and $\Exp
(\NN_m)\setminus\KK_m$ is contained in the sublevel set $\Lambda M^{<m^2\ell^2}$. From now on, with a slight abuse of notation we identify $\NN_m\equiv\Exp(\NN_m)$, and thus consider $\NN_m$ as a finite dimensional smooth submanifold of $\Lambda M$ containing $\KK^m$. If $0<\epsilon<\ell$, the inclusion induces the ring isomorphism 
\begin{align*}
H^*_{S^1}(\Lambda M^{< m^2(\ell+\epsilon)^2},\Lambda M^{< m^2\ell^2};\Q) 
& \cong
H^{*}_{S^1}(\NN_m,\NN_m\setminus\KK^m;\Q).
\end{align*}
Since the negative bundle $N^-\KK^m\to\KK^m$ is orientable, it has a non-zero Thom class with rational coefficients, which we can view as a relative cohomology class
\[\tau_m\in H^{\ind(\KK^m)}_{S^1}(\NN_m,\NN_m\setminus\KK^m;\Q),\] 
and we have a Thom isomorphism 
\begin{align*}
H^{*}_{S^1}(\NN_m;\Q) & \toup^{\cong} H^{*+\ind(\KK^m)}_{S^1}(\NN_m,\NN_m\setminus\KK^m;\Q),\\
\eta & \longmapsto \tau_m\smile\eta.
\end{align*}
The exponential map allows us to construct an $S^1$-equivariant deformation retraction of $\NN_m$ onto $\KK^m$, and, in particular, the inclusion induces the ring isomorphism 
\[H^*_{S^1}(\NN_m;\Q)\cong H^*_{S^1}(\KK^m;\Q).\] 
Since the $S^1$ action on $\KK^m$ is locally free, the projection 
\[\pr_1:\KK^m\times_{S^1} ES^1\to\KK^m/S^1,\qquad \pr_1(\gamma,e)=\gamma\] 
induces the ring isomorphism
\begin{align*}
\pr_1^*: H^*(\KK^m/S^1;\Q)\toup^{\cong} H^*_{S^1}(\KK^m;\Q).
\end{align*}
With a suitable orientation, this isomorphism maps the Euler class $e'$ of the circle bundle $\KK^m\to\KK^m/S^1$ to the Euler class $e_{\KK^m}$ of the circle bundle $\KK^m\times ES^1\to\KK^m\times_{S^1} ES^1$, i.e.,
\[e_{\KK^m}=\pr_1^*(e')\in H^2_{S^1}(\KK^m;\Q).\] It is well known that, in De Rham cohomology, $e'$ is represented by a symplectic form on $\KK^m/S^1$; see \cite{Boothby:1958ss}. In particular, $e_{\KK^m}^r\neq0$ if and only if $r\in\{0,...,n-1\}$, that is, $\NN_m$ has Fadell--Rabinowitz index 
\[\fr(\NN_m)=\fr(\KK^m)=n-1.\]

Since $E$ is perfect, the inclusion induces the surjective and, respectively, injective homomorphisms
\begin{align*}
 H^*_{S^1}(\NN_m,\NN_m\setminus\KK^m;\Q)
  \bssurjup^{i_m^*}
H^*_{S^1}(\Lambda M,\Lambda M^{< m^2\ell^2};\Q)
\eembup^{j_m^*}
 H^*_{S^1}(\Lambda M,M;\Q).
\end{align*}
In particular, there exists \[\nu_m\in H^{\ind(\KK^m)}_{S^1}(\Lambda M,\Lambda M^{< m^2\ell^2};\Q)\] 
such that $i_m^*(\nu_m)=\tau_m$ and $i_m^*(\nu_m\smile e_{\Lambda M}^{n-1})=\tau_m\smile e_{\NN^m}^{n-1}\neq0$. Therefore, we have
\begin{align*}
\alpha_m:=j_m^*(\nu_m)\neq0,\qquad 
\beta_m:=\alpha_m\smile e_{\Lambda M}^{n-1}=j_m^*(\nu_m\smile e_{\Lambda M}^{n-1})\neq0.
\end{align*}
Since $E$ is perfect, we readily see that $c_g(\alpha_m)= c_g(\beta_m)=m\ell$ for all integers $m\geq1$.

The following statement summarizes the discussion and proves that (iii) implies (ii) in Theorem~\ref{t:geodesics_Zoll}.

\begin{lem}
\label{l:Zoll_spectrum}
Let $M$ be a closed manifold of dimension at least $2$ admitting a Zoll Riemannian metric 
$g$. Then, $c_g(\alpha_m)=c_g(\beta_m)=m\ell$ for all $m\in \N$, where $\ell>0$ is the minimal period of the unit-speed closed geodesics of $(M,g)$.
\hfill\qed
\end{lem}

\begin{proof}[Proof of Theorem~\ref{t:geodesics_Zoll}]
We are only left to prove that (i) implies (iii), and that (i') implies (iii) when $M=S^n$ with $n\neq3$. The argument is analogous to the one of \cite[pages~22-23]{Mazzucchelli:2018pb}. The equality $m\ell=c_g(\alpha_m)=c_g(\beta_m)$, together with $\beta_m=\alpha_m\smile e_{\Lambda M}^{n-1}$ and Lemma~\ref{l:Besse}, implies that $g$ is Besse and $m\ell$ is a common period for its unit-speed closed geodesics. If $M=S^n$ with $n\neq3$, then the Berger conjecture \cite{Gromoll:1981kl, Radeschi:2017dz} implies that $g$ is Zoll, and by Lemma~\ref{l:Zoll_spectrum} the common period of the closed geodesics must be $\ell=c_g(\alpha_1)^{1/2}$. 

Assume now that $M$ is a general spin Zoll manifold, and $\ell=c_g(\alpha_1)=c_g(\beta_1)$. We have to prove that $\ell^2$ is the minimal positive critical value of $E$. The critical set $\KK:=\crit(E)\cap E^{-1}(\ell^2)$ is diffeomorphic to the unit tangent bundle $SM$ via the diffeomorphism $\gamma\mapsto(\gamma(0),\dot\gamma(0))$. Since the degree of the cohomology class $\alpha_1$ is $i(M)$, we have $\ind(\KK)\leq i(M)$.  
Hence $\ind(\KK)=i(M)$ because $i(M)$ is the minimal Morse index of the closed geodesics of $(M,g)$. Now, assume by contradiction that $g$ is not Zoll and so there exists $m>1$ such that the compact set
\begin{align*}
\WW:=\crit(E)\cap E^{-1}(m^{-2}\ell^2)
\end{align*}
is non-empty.
Since $i(M)\leq\ind(\WW)\leq\ind(\KK)$, we must have $\ind(\WW)=i(M)$. Since $E$ is perfect, we have
\begin{align*}
\dim H^{i(M)}_{S^1}(\Lambda M,M;\Q)
\geq
\dim (H^0(\WW/S^1)) + \dim(H^0(\KK/S^1)) \geq 2,
\end{align*}
which contradicts~\eqref{e:rank_one}.
\end{proof}

\bibliography{_biblio}
\bibliographystyle{amsalpha}

\end{document}